\documentclass[12pt]{amsart}
\usepackage{ifthen,amssymb,graphicx}
\usepackage[all]{xypic}
\usepackage[usenames]{color}
\usepackage{mathrsfs}

\newcommand{\catD}{{D}}
\newcommand{\dbc}[1]{{D}^b(#1)}
\newcommand{\dpc}[1]{{D}^+(#1)}
\newcommand{\dmc}[1]{{D}^-(#1)}

\newcommand{\cdbc}[1]{{D}^b_c(#1)}

\newcommand{\fmf}[3]{{\Phi^{#1}_{{\scriptscriptstyle #2\!\rightarrow\! #3}}}}

\newcommand{\what}[1]{{\widehat #1}}

\newcommand{\Ext}{{\operatorname{Ext}}}

\newcommand{\SExt}{{\mathcal{E}xt}}

\newcommand{\lotimes}{{\,\stackrel{\mathbf L}{\otimes}\,}}

\DeclareMathOperator{\Pic}{{Pic}}

\DeclareMathOperator{\Spec}{{Spec}}

\newcommand{\bbQ}{{\mathbb Q}}
\newcommand{\bbZ}{{\mathbb Z}}
\newcommand{\bbP}{{\mathbb P}}

\newcommand{\cE}{{\mathcal E}}
\newcommand{\cF}{{\mathcal F}}
\newcommand{\cG}{{\mathcal G}}
\newcommand{\calH}{{\mathcal H}}
\newcommand{\cI}{{\mathcal I}}

\newcommand{\calL}{{\mathcal L}}

\newcommand{\cM}{{\mathcal M}}

\newcommand{\cO}{{\mathcal O}}

\newcommand{\cT}{{\mathcal T}}

\newcommand{\bR}{{\mathbf R}}
\newcommand{\bL}{{\mathbf L}}

\newcommand{\cplx}[1]{{{\mathcal #1}^{\scriptscriptstyle\bullet}}}

\newcommand{\dSHom}[1]{{\mathcal{H}om_{#1}^{\scriptscriptstyle\bullet}}}

\newcommand{\rk}{\operatorname{rk}}

\newcommand{\marginnote}[1]{\ifthenelse{\isodd{\thepage}}{\normalmarginpar}
{\reversemarginpar}\marginpar{\fbox{\parbox{24mm}{\sloppy\footnotesize #1}}}}

\newcommand{\iso}{{\,\stackrel {\textstyle\sim}{\to}\,}}

\DeclareMathSymbol{\functor}{\mathbin}{AMSa}{"20}

\newtheorem{thm}{Theorem}[section]
\newtheorem*{thm*}{Theorem}
\newtheorem{cor}[thm]{Corollary}
\newtheorem{lem}[thm]{Lemma}
\newtheorem{prop}[thm]{Proposition}

\theoremstyle{definition}
\newtheorem{defin}[thm]{Definition}
\theoremstyle{remark}
\newtheorem{rem}[thm]{Remark}
\newtheorem{exe}[thm]{Example}

\numberwithin{equation}{section} 

\begin{document}

\title[Semistable sheaves on singular curves]{Moduli spaces of semistable sheaves on singular genus 1 curves}
\author[Hern\'andez Ruip\'erez]{Daniel Hern\'andez Ruip\'erez}
\author[L\'opez]{Ana Cristina L\'opez Mart\'{\i}n}
\author[S\'anchez]{Dar\'{\i}o S\'anchez G\'omez}
\author[Tejero]{Carlos Tejero Prieto}
\email{ruiperez@usal.es, anacris@usal.es, dario@usal.es, carlost@usal.es}
\address{Departamento de Matem\'aticas and Instituto Universitario de F\'{\i}sica Fundamental y Matem\'aticas
(IUFFyM), Universidad de Salamanca, Plaza de la Merced 1-4, 37008
Salamanca, Spain.}

\date{\today}
\thanks {Work supported by research projects MTM2006-04779 (MEC), SA001A07 (JCYL) and GR46 (JCYL). A.C.~L\'opez is also partially supported by a ``2007-Jos\'e
Castillejo'' grant of the Spanish Ministry for Education and Science.
D. ~S\'anchez acknowledges as well support from the European Union through the FP6
Marie Curie RTN \emph{ENIGMA} (Contract number MRTN-CT-2004-5652)} 
\subjclass[2000]{Primary: 18E30; Secondary:
14F05, 14J27, 14E30, 13D22, 14M05} \keywords{Geometric integral
functors, Fourier-Mukai, Cohen-Macaulay, fully faithful, elliptic
fibration, equivalence of categories, moduli, singular curves}

\begin{abstract}
 We find some equivalences of the derived category of coherent sheaves on a Gorenstein genus one curve that preserve the (semi)-stability of pure dimensional sheaves. Using them we establish new identifications between certain Simpson moduli spaces of semistable sheaves on the curve.  For rank zero, the moduli spaces are symmetric powers of the curve whilst for a fixed positive rank there are only a finite number of non-isomorphic spaces. We prove similar results for the relative semistable moduli spaces on an arbitrary genus one fibration with no conditions either on the base or on the total space. For a cycle $E_N$ of projective lines, we show that the unique degree 0 stable sheaves are the line bundles having degree 0 on every irreducible component and the sheaves $\mathcal{O}(-1)$ supported on one  irreducible component. We also prove that the connected component of the moduli space that contains vector bundles of rank $r$ is isomorphic to the $r$-th symmetric product of the rational curve with one node.
\end{abstract}

\maketitle



\section*{Introduction}
Elliptic fibrations have been used in string theory, notably in
connection with mirror symmetry on Calabi-Yau manifolds and
D-branes. The study of relative moduli spaces of semistable sheaves
on elliptic fibrations, aside from its mathematical importance,
provides a geometric background to string theory. In the case of
integral elliptic fibrations, a complete description is  already
known and among the papers considering the problem we can cite
\cite{BBHM98,Bri98,BrM02,HMP02}.
A study of these relative spaces for a more general class of genus
one fibrations (for instance, with non-irreducible fibers and even
singular total spaces) turns out to be an interesting problem.

On the one hand, for sheaves of rank 1 a fairly complete study of a
class of these moduli spaces (compactified relative Jacobians),
including those associated to relatively minimal elliptic surfaces,
was carried out by one of the authors in \cite{LM05,LM05a} (see also \cite{Ca00, Ca2}). On the
other hand, nowadays it is well understood the efficient key idea of
the ``spectral cover construction'' discovered for the first time by
Friedman-Morgan-Witten in \cite{FMW99} and widely used later by many
authors. The method shows how useful is the theory of integral
functors and Fourier-Mukai transforms in the problem. The study
developed by two of the authors in \cite{HLS07} and \cite{HLS08} on
relative integral functors for singular fibrations gives a new
insight in this direction.

From the results in that paper one gets new information about
moduli spaces of relative semistable sheaves of higher rank for a
{\it genus one fibration} $p\colon S\to B$, that is, a projective
Gorenstein morphism whose fibers are curves of arithmetic genus one
and trivial dualizing sheaf but without further assumptions on $S$
or $B$.
The fiber of the relative moduli
space over a point $b\in B$ is just the absolute moduli space of
semistable sheaves on $S_b$, so that in order to start with the
relative problem one has to know in advance the structure of the
absolute moduli spaces for the possible degenerations of an elliptic
curve. There are some cases where the structure of the singular
fibers is known. For smooth elliptic surfaces over the complex
numbers, the classification was given by Kodaira  \cite{kod} and for
smooth elliptic threefolds over a base field of characteristic
different from 2 and 3, they were classified by Miranda
\cite{Mir83}. In both cases, the possible singular fibers are plane
curves of the same type, the so-called Kodaira fibers. Nevertheless,
in a genus one fibration non-plane curves can appear as
degenerated fibers. So that our genus one fibrations may have
singular fibers other than the Kodaira fibers. The study of the moduli spaces of vector bundles on smooth elliptic curves dates back to Atiyah \cite{At57} and Tu \cite{Tu94a}, who proved that for an elliptic curve $X$ there is an isomorphism $\cM(r,d)\equiv Sym^{m} X$, where $m=\gcd(r,d)$, between the moduli space of semistable sheaves of rank $r$ and degree $d$ and the symmetric product of the curve. A very simple way to prove this isomorphism is by using Fourier-Mukai transforms (cf.~\cite{Pol03,HePl05}). This method has been generalized to irreducible elliptic curves (i.e., rational curves with a simple node or cusp) in \cite[Chapter 6]{BBH08}) obtaining that $\cM(r,d)\equiv Sym^{m} X$, where $m=\gcd(r,d)$ also in this case.
In the case of singular curves, the moduli spaces of semistable torsion free sheaves were first constructed and studied by Seshadri \cite{Ses82}; his construction can now be seen as a particular case of the general construction of the moduli spaces of semistable pure sheaves due to Simpson \cite{Simp96a}. The properties of these moduli spaces and their degeneration properties have been studied by many authors (see, for instance, \cite{Se00,NaSe97,NaSe99,Ca00, Ca2,LM05,LM05a}).

The paper is divided in two parts.  In the first part, we
consider $X$ an arbitrary Gorenstein genus one curve with trivial
dualizing sheaf. The group of all integral functors that are exact
autoequivalences of $D_c^b(X)$ is still unknown and  a criterion
characterizing those Fourier- Mukai transforms that preserve
semistability for a non-irreducible curve of arithmetic genus 1
seems to be a difficult problem.  Here we find some equivalences of
its derived category $D_c^b(X)$ of coherent sheaves that preserve
the (semi)-stability of pure dimensional sheaves. One is  given by
the ideal of the diagonal and the other is provided by twisting by
an ample line bundle (see Theorem \ref{t:preservation}).  Our proof
follows the ideas in \cite{Bri98,Pol03} where the result was proved for a
smooth elliptic curve and in \cite{BBH08} for an irreducible singular curve of arithmetic genus one. The results of this section allow to ensure
that for rank zero, the moduli spaces are the symmetric powers of
the curve whilst for a fixed positive rank there are only a finite number of
non-isomorphic moduli spaces (see Corollary \ref{c:isom}). Unlike
the case of a smooth curve, moduli spaces of semistable
sheaves on a curve with many irreducible components are not
normal (even in the case where the rank is 1). Its structure depends
very strongly on the particular configuration of every single curve.
The difficulty in determining the stability conditions for a sheaf in
this case points out the relevance of the identifications of
Corollary \ref{c:isom}. In fact, for a curve with two irreducible
components endowed with a polarization of minimal degree, they
reduce the study either to the case of rank 0 or degree 0. Coming
back to the relative case, the section finishes with Corollary
\ref{c:isom2} which establishes new identifications between certain
relative Simpson moduli spaces of (semi)stable sheaves for a genus
one fibration.

In the second part, we focus our study in a curve of type $E_N$ and
in the case of degree 0 which is particularly interesting as in this
case  semistability does not depend on the polarization. Proposition
\ref{p:grupoK} computes the Grothendiek group of coherent sheaves
for any reduced connected and projective curve whose  irreducible
components are isomorphic to $\mathbb{P}^1$.
The discrete invariants corresponding to the
Grothendieck group behave well with respect to Fourier-Mukai
transforms and are important tools for the analysis of the moduli spaces.
Although a description of all torsion-free sheaves on a
cycle of projective lines $E_N$ is known, as we mentioned above, it
is by no means a trivial problem to find out which of them are
semistable. For instance, contrary to what happens for an elliptic
curve, semistability is not guaranteed by the simplicity of the
sheaf. For $E_1$, that is, a rational curve with one node, this was
done in \cite{BuKr04} for the degree zero case and in \cite{BuKr05}
otherwise. Using the description of indecomposable torsion-free
sheaves on $E_N$ given in \cite{BBDG07} and the study of semistable
torsion-free sheaves on $E_N$ and on tree-like curves of
\cite{LM05}, Theorem \ref{t:stables} proves that a degree 0
stable sheaf is either a line bundle having degree 0 on
every irreducible component of $X$ or $\cO_{C_i}(-1)$ for some
irreducible component $C_i$. Then Corollary \ref{c:JHfactors} gives
the possible Jordan-Holder factors of any degree 0 semistable sheaf.
In the integral case, if the sheaf is indecomposable all
Jordan-H\"older factors are isomorphic to each other. This is no longer the case for cycles of projective lines. Proposition \ref{p:graded} computes the graded object of any indecomposable semistable sheaf of degree 0.
The structure of the connected component of
the moduli space that contains vector bundles of rank $r$ is given
in Theorem \ref{t:sym}. Namely, it is isomorphic to the $r$-th
symmetric product $Sym^r E_1$ of the rational curve with one node.
Having studied the case of degree zero, the results of the first
part of the paper allow to cover other cases (see Remark
\ref{r:othercases}). In particular, the connected component of
the moduli space that contains vector bundles of rank $r$ and degree $rh$, where $h$ is the degree of the polarization, is also isomorphic to $Sym^r E_1$.

In this paper, scheme means separated  scheme of finite type over an
algebraically closed field $k$ of characteristic zero. By a
Gorenstein or a Cohen-Macaulay morphism, we understand a \emph{flat}
morphism of schemes whose fibers are respectively Gorenstein or
Cohen-Macaulay. For any scheme $X$ we denote by $\catD(X)$ the
derived category of complexes of $\cO_X$-modules with quasi-coherent
cohomology sheaves. This is the essential image of the derived
category of quasi-coherent sheaves in the derived category
$\catD(\mathfrak{Mod}(X))$ of all $\cO_X$-modules
\cite[Corollary~5.5]{BoNee93}. Analogously $\dpc{X}$, $\dmc{X}$ and
$\dbc{X}$ denote the derived categories of complexes which are
respectively bounded below, bounded above and bounded on both sides,
and have quasi-coherent cohomology sheaves. The subscript $c$ will
refer to the corresponding subcategories of complexes with coherent
cohomology sheaves. By a point we always mean a closed point. As it
is usual, if $x\in X$ is a point, $\mathcal{O}_x$ denotes the
skyscraper sheaf of length 1 at $x$, that is, the structure sheaf of
$x$ as a closed subscheme of $X$, while the stalk of $\mathcal{O}_X$
at $x$ is denoted $\mathcal{O}_{X,x}$.

\subsection*{Acknowledgements} We thank I. Burban for pointing out a mistake in the first version of this paper and for showing us the example of a simple not WIT sheaf described in Remark \ref{rem:simplenotwit}. We also thank C.S. Seshadri for drawing to our attention Strickland's result \cite{Str82} which implies Lemma \ref{lem:reduced}, and U.N. Bhosle for pointing out some inaccuracies and mistakes.
The second and the third author
would like to thank respectively the Warwick Mathematics Institute and the
Mathematical Institute of Oxford for hospitality and very stimulating
atmosphere whilst this paper was written.

\section{Fourier-Mukai transforms preserving stability}
\subsection{A non-trivial  Fourier-Mukai transform on genus one curves}

Let $X$ and $Y$ be proper schemes. We denote the two projections of
the direct product $X\times Y$ to $X$ and $Y$ by $\pi_X$ and
$\pi_Y$.

Let $\cplx{K}$ be an object in $\cdbc{X\times Y}$. The integral
functor of kernel $\cplx{K}$ is the functor $\fmf{\cplx K}{X}{Y}
\colon \catD (X) \to \catD (Y)$ defined as
$$\fmf{\cplx K}{X}{Y}(\cplx{F})=\bR \pi_{Y\ast}(
\pi_X^\ast\cplx{F}\lotimes \cplx{K})$$ and it maps $\dmc{X}$ to
$\dmc{Y}$.

In order to determine whether an integral functor maps bounded
complexes to bounded complexes, the following notion was introduced
in \cite{HLS07}.

\begin{defin} Let $f\colon Z \to T$ be a morphism of
schemes. An object $\cplx E$ in $\cdbc{Z}$ is said to be of
\emph{finite homological dimension over $T$}  if $\cplx E\lotimes
\bL f^\ast \cplx G$ is bounded for any $\cplx{G}$ in $\cdbc{T}$.
\end{defin}

The proof of the following lemma can also be found in
\cite[Proposition~2.7]{HLS08}.

\begin{lem} \label{p:kernelfhd}
Assume that $X$ is a projective scheme and let $\cplx{K}$ be an
object in $\cdbc{X\times Y}$. The functor $\fmf{\cplx{K}}{X}{Y}$
maps $\cdbc{X}$ to $\cdbc{Y}$ if and only if $\cplx{K}$ has finite
homological dimension over $X$. \end{lem}

Let us suppose that $X$ is a projective Gorenstein curve with
arithmetic genus $\dim H^1(X, \cO_{X})=1$ such that its dualizing
sheaf is trivial. This includes all the so-called Kodaira fibers, that
is, all singular fibers of a smooth elliptic surface over the
complex numbers (classified by Kodaira in \cite{kod}) and of a
smooth elliptic threefold over a base field of characteristic
different from 2 and 3 (classified by Miranda in \cite{Mir83}). In
these two cases, all fibers are plane curves.
Here, we do not
need to assume that our curve $X$ is a plane curve.
Notice also that an irreducible curve of arithmetic genus one has always trivial dualizing sheaf, but this is no longer true for reducible curves. Therefore in \cite{HLS07} a genus one fibration is defined as a projective Gorenstein morphism $p\colon S\to B$ whose fibers have arithmetic genus one and trivial dualizing sheaf.

Using the theory of spherical objects by Seidel and Thomas
\cite{SeTh01}, we have the following

\begin{prop} Let $X$ be a projective Gorenstein curve with
arithmetic genus 1 and trivial dualizing sheaf. Let $\cI_\Delta$ be
the ideal sheaf of the diagonal immersion $\delta\colon
X\hookrightarrow X\times X$. One has:
\begin{enumerate}
\item The ideal sheaf $\cI_\Delta$ is an object in $\cdbc{X\times
X}$ of finite homological dimension over both factors.
\item The functor $\Phi=\fmf{\cI_{\Delta}}{X}{X}\colon \cdbc{X}\to
\cdbc{X}$ is an equivalence of categories.
\item The integral functor
$\widehat{\Phi}=\fmf{\cI_{\Delta}^\ast}{X}{X}\colon \cdbc{X}\to
\cdbc{X}$, where $\cI_{\Delta}^\ast$ is the dual sheaf, is a shift of
the quasi-inverse of $\Phi$ with $\widehat{\Phi}\circ \Phi\simeq
[-1]$ and $\Phi\circ \widehat\Phi\simeq [-1]$.
\end{enumerate}
\end{prop}
\begin{proof}
(1) Denote by $\pi_i\colon X\times X\to X$ with $i=1,2$ the two
projections. By the symmetry, to see that $\cI_\Delta$ is of finite
homological dimension over both factors it is enough to prove it
over the first one. Using the exact sequence $$0\to \cI_\Delta\to
\cO_{X\times X}\to \delta_\ast \cO_X\to 0\, ,$$ it suffices to see
that $\delta_\ast \cO_X$ has finite homological dimension over the
first factor. We have then to prove that for any bounded complex
$\cplx{F}$  on $X$, the complex $\delta_\ast \cO_X\otimes \pi_1^\ast
\cplx{F}$ is also a bounded complex and this follows from the
projection formula for $\delta$.

(2) Since $X$ is a projective Gorenstein curve of genus one and
trivial dualizing sheaf, $\cO_X$ is a spherical object of
$\cdbc{X}$. By \cite{SeTh01} the twisted functor $T_{\cO_{X}},$
along the object $\cO_X$, is an equivalence of categories. Since
$\Phi\simeq T_{\cO_{X}}[-1]$, the statement follows.

(3) By \cite[Proposition~2.9]{HLS08}, the functor
$\Phi^{\cI_{\Delta}^\vee[1]}_{X\to X}$ is the right adjoint to
$\Phi$ where $\cI_{\Delta}^\vee=\bR \dSHom{\cO_{X\times
X}}(\cI_{\Delta},\cO_{X\times X})$ is the dual in the derived
category, then it is enough to prove that $\cI_{\Delta}^\vee$ is
isomorphic to $\cI_{\Delta}^\ast$, the ordinary dual. Indeed, one
has to check that $\SExt^i_{\cO_{X\times X}}(\cI_{\Delta},
\cO_{X\times X})=0$ for $i\ge 1$. Let us consider the exact sequence
$$
0\to \cI_{\Delta} \to \cO_{X\times X} \to \delta_\ast\cO_{X}\to 0 \,
.$$ Taking local homomorphisms in $\cO_{X\times X}$, we get an exact
sequence
\begin{equation*}
0\to\cO_{X\times X} \to \cI_{\Delta}^\ast \to \SExt^1_{\cO_{X\times
X}}(\delta_\ast\cO_{X}, \cO_{X\times X}) \to 0
\end{equation*}
and isomorphisms
$$\SExt^{i-1}_{\cO_{X\times X}}(\cI_{\Delta}, \cO_{X\times X}) \simeq  \SExt^{i}_{\cO_{X\times X}}(\delta_\ast\cO_{X}, \cO_{X\times X})
\qquad \text{for all $i>1$,}
$$
which proves our claim because $X$ is Gorenstein.
\end{proof}

We shall use the following notation: for any integral functor
$ \mathbf{S}\colon \cdbc{X}\to \cdbc{X}$, $ \mathbf{S}^j$ denotes the $j$-th
cohomology sheaf of $ \mathbf{S}$, unless confusion can arise. Remember
that a sheaf $\cE$ on $X$ is said to be WIT$_i$-$ \mathbf{S}$ if
$ \mathbf{S}(\cE)\simeq \mathbf{S}^i(\cE)[-i]$. In this case, we denote the unique
non-zero cohomology sheaf $ \mathbf{S}^i(\cE)$ by $\widehat{\cE}$.

Note that in our particular situation $ \mathbf{S}=\Phi$, since $\cI_{\Delta}$ is flat
over the first factor and the fibers of $\pi_2$ are of dimension
one, for any sheaf $\cE$ on $X$ one has $\Phi^j(\cE)=0$ unless
$0\leq j\leq 1$.

We now adapt to our case some well-known properties about
WIT sheaves.

\begin{prop}\label{prop:properties RFM}
The following results hold:
\begin{enumerate}
 \item There exists a Mukai spectral
sequence
\begin{equation*}
E_2^{p,q}=\widehat\Phi^p(\Phi^q(\cE)) \implies
\begin{cases}
\cE & \text{if $p+q=1$} \\
0 & \text{otherwise\,.}
\end{cases}
\end{equation*}
\item Let $\cE$ be a
WIT$_i$-$\Phi$ sheaf on $X$. Then $\widehat{\cE}$ is a
WIT$_{1-i}$-$\widehat{\Phi}$ sheaf on $X$ and
$\widehat{\what{\cE}\,}=\cE.$
\item For every sheaf $\cE$ on $X$, the sheaf $\Phi^0(\cE)$ is
WIT$_1$-$\widehat\Phi,$ while the sheaf $\Phi^1(\cE)$ is
WIT$_0$-$\widehat\Phi$.

\item There exists a short exact sequence
$$0\longrightarrow\what\Phi^1(\Phi^0(\cE))\longrightarrow\cE\longrightarrow\widehat\Phi^0(\Phi^1(\cE))\longrightarrow 0\,.$$
\end{enumerate}
\end{prop}
\begin{proof}
(1) and (2) follow from \cite[Eq.~2.35 and Prop.~2.34]{BBH08}. (3) is a direct consequence of (1) and (4) is the exact sequence of lower terms of the Mukai spectral sequence.
\end{proof}
\subsection{Preservation of the absolute stability for some equivalences}

\subsubsection{Pure sheaves and Simpson stability}
A notion of stability and semistability for pure sheaves on a
projective scheme with respect to an ample divisor was given by Simpson in \cite{Simp96a}. He also proved the existence of the corresponding coarse moduli spaces.

Let $X$ be a projective scheme of dimension $n$ over an
algebraically closed field $k$ of characteristic zero and fix $H$ a
polarization, that is, an ample divisor on $X$. For any coherent
sheaf $\cE$ on $X$, denote $\cE(sH)=\cE\otimes \cO_X(sH)$.

The {\it Hilbert polynomial} of $\cE$ with respect to $H$ is defined
to be the unique polynomial $P_{\cE}(s)\in \mathbb{Q}[s]$ given by
$$P_{\cE}(s)=h^0(X, \cE(sH)) \text{ for all } s\gg 0\, .$$ This
polynomial has the form
$$P_{\cE}(s)=\frac{r(\cE)}{m!}s^m+\frac{d(\cE)}{(m-1)!}s^{m-1}+\dots$$
where $r(\cE)\geq 0$ and $d(\cE)$ are integer numbers and its degree
$m\leq n$ is equal to the dimension of the support of $\cE$.

\begin{defin}
A coherent sheaf $\cE$ is \emph{pure} of dimension $m$ if the
support of $\cE$ has dimension $m$ and the support of any nonzero
subsheaf $0\to\cF\to\cE$ has dimension $m$ as well.
\end{defin}

When $X$ is integral, pure sheaves of dimension $n$ are precisely
torsion-free sheaves. We can then adopt the following definition.
\begin{defin}
A coherent sheaf $\cE$ on $X$ is \emph{torsion-free} if it is pure
of dimension $n=\dim X,$ and it is a \emph{torsion} sheaf if the
dimension of its support is $m<n.$
\end{defin}

When $X$ is a projective curve with a fixed polarization $H$, the
Hilbert polynomial of a coherent sheaf $\cE$ on $X$ is then
$$P_{\cE}(s)=r(\cE)s+d(\cE)\in \mathbb{Z}[s],$$ a polynomial with integer
coefficients and at most of degree one. It is constant precisely for
torsion sheaves.

The (Simpson) {\it slope} of $\cE$ is defined as
$$\mu_S(\cE)=\frac{d(\cE)}{r(\cE)}\, .$$ This is a rational number if $r(\cE)\neq 0$ and
it is equal to infinity for torsion sheaves. It allows us to define
(Simpson) $\mu$-stability and $\mu$-semistability for pure sheaves
as usual.
\begin{defin}
A sheaf $\cE$ on $X$ is (Simpson) $\mu_S$-stable (resp.~$\mu_S$-semistable) with respect to $H$, if it is pure and for every
proper subsheaf $\cF\hookrightarrow\cE$ one has
$\mu_S(\cF)<\mu_S(\cE)$ (resp.~$\mu_S(\cF)\leq\mu_S(\cE)$).
\end{defin}

With these definitions any torsion sheaf on $X$ is
$\mu_S$-semistable and it is $\mu_S$-stable if and only if it has no
proper subsheaves, that is, it is isomorphic to $\cO_x$, the
structure sheaf of a point $x\in X$. As a particular case of
Simpson's work \cite{Simp96a}, we have the following existence
result. Fixing a polynomial $P(s)=rs+d\in \mathbb{Z}[s]$, and a
polarization $H$ on $X$, if the class of $\mu_S$-semistable sheaves
on $X$, with respect to $H$, with Hilbert polynomial equal to $P$ is
non-empty, then it has a coarse moduli space $\cM_X(r,d)$ which is
a projective scheme over $k$. Rational points of $\cM_X(r,d)$
correspond to $S$-equivalence classes of $\mu_S$-semistable sheaves
with Hilbert polynomial $P(s)=rs+d$.

\begin{rem}\label{r:RR} When $X$ is an integral curve and $\cE$ is a coherent sheaf
on it, one has classical notions of rank of $\cE$, as the rank at
the generic point of $X$, and degree of $\cE$, as
$\chi(\cE)-\rk(\cE)\chi(\cO_X)$. The Riemann-Roch theorem gives us
what is the relation between the coefficients of the Hilbert
polynomial and the usual rank and degree of $\cE$, namely
\begin{align*}
r(\cE)&=\deg(X)\cdot\rk(\cE)\\
d(\cE)&=\deg(\cE)+\rk(\cE)\cdot\chi(\cO_X)
\end{align*}
where $\deg(X)$ is the degree of $X$ defined in terms of the
polarization $H$. In this case, the Simpson notions of
$\mu_S$-stability and $\mu_S$-semistability are equivalent to the
usual ones for torsion-free sheaves. Thus, for integral curves
$\mu_S$-semistability does not depend on the polarization. This is
no longer true for non-integral curves (see \cite{LM05} for more
details).
\end{rem}
\subsubsection{Invariants of the transforms and the WIT condition}

In the rest of this section we assume that $X$ is a projective
Gorenstein curve of arithmetic genus one with trivial dualizing sheaf
and that $H$ is a fixed polarization of degree $h$ on $X$.

Since the curve $X$ may be a singular curve, we will work with the
Hilbert polynomial of a sheaf instead of its Chern characters that
might not be defined. The following proposition computes the Hilbert
polynomial of the transform of $\cE$ by the equivalences $\Phi$ and
$\widehat{\Phi}$ of the previous subsection and by
$\Psi=\fmf{\delta_\ast \cO_X(H)}{X}{X}$ and
$\what{\Psi}=\fmf{\delta_\ast \cO_X(-H)}{X}{X}$. Remember that for a
bounded complex $\cplx{F}$, the Euler characteristic is defined to
be the alternate sum
$$\chi(\cplx{F})=\sum_i(-1)^i\chi({\calH}^i(\cplx{F}))\, .$$ and the Hilbert
polynomial is by definition $P_{\cplx{F}}(s)=\chi(\cplx{F}(sH))$.

\begin{prop}\label{prop:RR}
Let $\cE$ be a sheaf on $X$ with Hilbert polynomial
$P_{\cE}(s)=rs+d$. Then
\begin{enumerate}
\item The Hilbert
polynomial of the complex $\Phi({\cE})$ (resp.~$\widehat{\Phi}({\cE})$) is equal to $(dh-r)s-d$ (resp.~$(dh+r)s+d$).
\item The Hilbert polynomial of the sheaf $\Psi(\cE)$ (resp.~$\what{\Psi}(\cE)$) is equal to $rs+d+r$ (resp.~$rs+d-r$).
\end{enumerate}
\end{prop}
\begin{proof} (1) Denote $\cO=\cO_{X\times X}$ and consider the exact sequence $$0\to
\cI_\Delta\to \cO\to \delta_\ast \cO_X\to 0\, .$$ In the derived
category $\cdbc{X}$, this induces an exact triangle
\begin{equation}\label{e:triangle}
 \Phi(\cE)\to \fmf{\cO}{X}{X}(\cE)\to \cE\to \Phi(\cE)[1]
\end{equation}
 for any sheaf $\cE$ on $X$. Since the Euler characteristic is
additive for exact triangles in the derived category, the Hilbert
polynomial of the complex $\Phi({\cE})$ is equal to
$$\chi(\fmf{\cO}{X}{X}(\cE)(sH))-(rs+d)\, .$$

If  $p\colon X\to \Spec k$ is the projection of $X$ onto a point and
$\pi_i\colon X\times X\to X$  are the natural projections, the
base-change formula for the diagram
$$\xymatrix{X\times X\ar[r]^{\pi_1}\ar[d]_{\pi_2}& X\ar[d]^p\\X\ar[r]_p & \Spec k}$$ shows
that
\begin{equation}\label{e:dersect}
\fmf{\cO}{X}{X}(\cE)=\bR {\pi_2}_\ast(
\pi_1^\ast\cE)\simeq p^\ast\bR p_\ast(\cE)\simeq p^\ast \bR\Gamma(X,
\cE)=\bR \Gamma(X, \cE)\otimes_k\cO_X\, .
\end{equation}
 Then
$\chi(\fmf{\cO}{X}{X}(\cE)(sH))=\chi(\bR\Gamma(X,\cE)\otimes_k\cO_X(sH))=\chi(\cE)\chi(\cO_X(sH))=d(sh)$,
and the result follows. The Hilbert polynomial of $\what{\Phi(\cE)}$ is computed using that $\Phi\circ \widehat{\Phi}\simeq
[-1]$.

(2) Since the equivalence $\Psi$ (resp.~$\what{\Psi})$ is given  by
twisting by the line bundle $\cO_X(H)$ (resp.~$\cO_X(-H)$), this
part is immediate.
\end{proof}

\begin{rem}\label{r:linearity} Notice that if $\calL$ is an arbitrary line bundle on $X$, the second coefficient of the Hilbert polynomial of the transform $\fmf{\delta_\ast \calL}{X}{X}(\cE)$ is not  in general a linear function of the coefficients $r$ and $d$ (see Example \ref{e:nonlinear}).
\end{rem}

Any sheaf $\cE$ on $X$ is  WIT$_0$-$\Psi$ and WIT$_0$-$\what{\Psi}$.
In order to prove the preservation of stability under the
equivalence of $\cdbc{X}$ defined by the ideal of the diagonal
$\cI_{\Delta}$, we shall need a description of semistable sheaves
WIT$_i$ with respect to the Fourier-Mukai transforms $\Phi$ and
$\what{\Phi}$.

\begin{cor}\label{c:Wit}
Let $\cE$ be a non-zero sheaf on $X$.
\begin{enumerate}
\item If $\cE$ is WIT$_0$-$\Phi$, then $\mu_S(\cE)>1/h$.
\item If $\cE$ is WIT$_1$-$\Phi$, then  $\mu_S(\cE)\leq 1/h$.
\item If $\cE$ is WIT$_0$-$\what\Phi$, then  $\mu_S(\cE)>-1/h$.
\item If $\cE$ is WIT$_1$-$\what\Phi$, then  $\mu_S(\cE)\leq -1/h$
\end{enumerate}
\end{cor}
\begin{proof}
Let $P_{\cE}(s)=rs+d$ the Hilbert polynomial of $\cE$. By
Proposition \ref{prop:RR}, the Hilbert polynomial of $\Phi(\cE)$ is
$P_{\Phi(\cE)}(s)=(dh-r)s-d$. Suppose that $\cE$ is WIT$_0$-$\Phi$.
Then $\what{\cE}=\Phi(\cE)$ and $dh\geq r$. If $dh=r$, the transform
$\what\cE$ is a torsion sheaf of length $-d=-r/h\leq 0$ and thus it
is equal to zero.  But this is absurd because $\cE$ is non-zero. The
second statement follows straightforwardly and the proof for
$\what\Phi$ is similar.
\end{proof}

\begin{rem} \label{r:propert} The following easy properties will be used in the rest of
the section:
\begin{enumerate}
\item Torsion sheaves on $X$ are WIT$_0$ with respect to both
equivalences $\Phi$ and $\what\Phi.$ \item If a sheaf $\cE$ is
WIT$_1$ with Hilbert polynomial $P_{\cE}(s)=rs+d$, then $r\neq 0$.
\end{enumerate}
\end{rem}

\begin{prop}\label{prop:WIT classification}
If $\cE$ is a $\mu_S$-semistable  sheaf on $X$, then
\begin{enumerate}
\item $\cE$ is  WIT$_0$-$\Phi$  if and only if $\mu_S(\cE)>1/h$.
\item $\cE$ is  WIT$_1$-$\Phi$  if and only if $\mu_S(\cE)\leq
1/h$.
\end{enumerate}
\end{prop}
\begin{proof} If $\cE$ is a torsion sheaf, the result follows from
Remark \ref{r:propert}. Suppose then that $\cE$ is torsion-free and
consider the short exact sequence
$$0\longrightarrow\what\Phi^1(\Phi^0(\cE))\longrightarrow\cE\longrightarrow\widehat\Phi^0(\Phi^1(\cE))\longrightarrow 0\, .$$

(1) The direct implication is given by Corollary \ref{c:Wit}. Let us
prove the converse. If $\mu_S(\cE)>1/h$ and $\cE$ is not
WIT$_0$-$\Phi$, by the above exact sequence
$\what\Phi^0(\Phi^1(\cE))$ is a non-zero quotient of $\cE$ and
WIT$_1$-$\Phi$. By Corollary \ref{c:Wit}, its slope is
$\mu_S(\what\Phi^0(\Phi^1(\cE)))\leq 1/h$  and consequently,
$\mu_S(\what\Phi^1(\Phi^0(\cE))) > \mu_S(\cE )$. This contradicts
the semistability of $\cE$. Thus, $\cE$ is WIT$_0$-$\Phi$.

\medskip
(2) The direct implication follows again from Corollary \ref{c:Wit}.
For the converse we proceed as before. If $\mu_S(\cE)\leq 1/h$ and
$\cE$ is not WIT$_1$-$\Phi$, by the exact sequence
$\what\Phi^1(\Phi^0(\cE))$ is a non-zero subsheaf of $\cE$ and
WIT$_0$-$\Phi$. By Corollary \ref{c:Wit}, its slope is
$\mu_S(\what\Phi^1(\Phi^0(\cE)))> 1/h$ and this contradicts the
semistability of $\cE$. Thus, $\cE$ is WIT$_1$-$\Phi$.

\end{proof}
There exists a similar result for $\what\Phi$ whose proof is
analogous.
\begin{prop}\label{prop:WIT classification2}
If $\cE$ is a $\mu_S$-semistable  sheaf on $X$, then
\begin{enumerate}
\item $\cE$ is  WIT$_0$-$\what\Phi$  if and only if
$\mu_S(\cE)>-1/h$. \item $\cE$ is  WIT$_1$-$\what\Phi$  if and only
if $\mu_S(\cE)\leq -1/h$.
\end{enumerate}
\end{prop}

The following simple but useful result is known as the
Parseval theorem (see for instance \cite{BBH08} for a proof).
\begin{prop}\label{prop:Parseval}
Let $\Phi\colon D(X)\to D(Y)$ be an exact fully faithful functor,
$\cF$ a WIT$_i$-$\Phi$ sheaf and $\cG$ a WIT$_j$-$\Phi$ sheaf on
$X.$ Then for all $k,$ one has
$$\Ext_X^{k}(\cF,\cG)\simeq\Ext_Y^{k+i-j}(\what\cF,\what\cG)$$
In particular if $\cF$ is a simple WIT-$\Phi$ sheaf, then the
transform $\what\cF$ is also simple. \end{prop}

\begin{prop}\label{p:simple}
Let $\cE$ be a simple (resp.~indecomposable) semistable sheaf on $X.$ Then
 the transform $\what\cE$ with respect to both  $\Phi$ and
$\what\Phi$ is also a
simple (resp.~indecomposable) sheaf.
\end{prop}
\begin{proof}
By Propositions \ref{prop:WIT classification} and \ref{prop:WIT classification2}, $\cE$ is WIT with respect to both $\Phi$ and $\what\Phi$. Then $\what\cE$ is indecomposable when $\cE$ is so. Moreover,  if $\cE$ is simple, $\what\cE$ is simple by the Parseval  formula (Proposition \ref{prop:Parseval}).
\end{proof}

\subsubsection{Preservation of (semi)stability} If $X$ is an irreducible curve of arithmetic genus
one, the group of exact auto-equivalences of its derived category
$\cdbc{X}$  is described in \cite{BuKr05}. As it happens for the
smooth case, this group is generated by the trivial equivalences
(twists by line bundles on $X$, automorphisms of $X$ and the shift
functor [1]) together with the Fourier-Mukai transform $\Phi$ whose
kernel is the ideal of the diagonal. Then, taking into account that
on integral curves tensoring by line bundles preserves trivially the
(semi)stability of sheaves, the fact that the non-trivial
Fourier-Mukai functor $\Phi$ transforms (semi)stable sheaves into
(semi)stable sheaves (up to shift) and stable sheaves into stable ones (also up to shift)
(cf.~\cite{BBH08}) ensures that any auto-equivalence of  the derived
category $\cdbc{X}$ preserves stability.

However, this is no longer true for non-irreducible curves.
Actually, if $X$ is a non-irreducible curve, there are examples of
equivalences of $\cdbc{X}$ that do not preserve semistability, and
we can find examples of such equivalences among those of the most
simple type, namely, among the equivalences $\fmf{\delta_\ast
\calL}{X}{X}$ consisting of twisting by a line bundle $\calL$.

Consider, for instance, a curve $X$  of type $E_2$, that is, two
rational curves meeting transversally at two points (cf.~Figure
\ref{fig:e2}).

\begin{figure}[htbp]
\centering
\includegraphics[width=.3in]{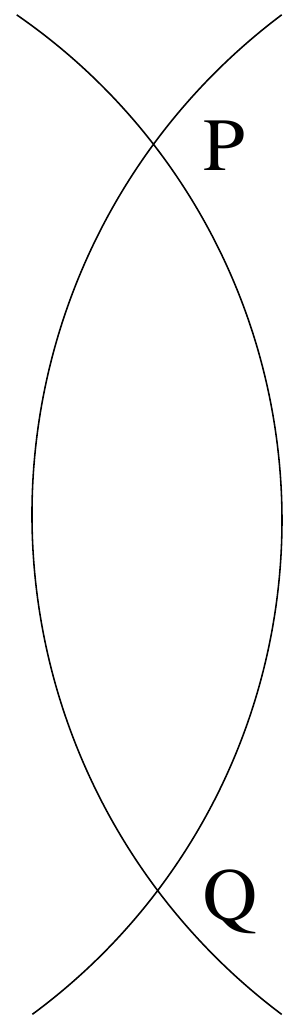}
\caption{The curve $E_2$} \label{fig:e2}
\end{figure}

Take  a line bundle $\calL$ on $X$ which has degree $2$ in one
irreducible component and degree $-2$ in the other one. From Propositions
6.2 and 6.3 in \cite{LM05},  $\cO_X$ is a stable sheaf but $\calL$
is not even semistable. Then twisting by $\calL$ is a Fourier-Mukai
transform  which does not preserve semistability.

\begin{rem}\label{rem:simplenotwit}
A straightforward computation shows that both $H^0(X, \calL)$ and $H^1(X, \calL)$ are one dimensional vector spaces. Using Equations \eqref{e:triangle} and \eqref{e:dersect} we deduce that $\Phi(\cE)$ is a complex with two nonzero cohomology sheaves. This proves that simple (unstable) sheaves on $X$ may fail to be WIT.
\end{rem}

Thus it is important  to characterize the auto-equivalences of
$\cdbc{X}$ which preserve stability on a non-irreducible curve $X$
of arithmetic genus 1. This seems to be a difficult task, and here
we just provide non-trivial instances of such equivalences.

To begin with, note that twisting by the ample sheaf $\cO_X(H)$
trivially preserves stability, that is,  the transform of a
$\mu_S$-(semi)stable sheaf by the equivalences $\Psi$ or
$\what{\Psi}$ is again $\mu_S$-(semi)stable (cf.~Proposition
\ref{prop:RR}).

In this section we prove that the non-trivial Fourier-Mukai functors
$\Phi$ and $\what{\Phi}$ preserve semistability as well.

\begin{lem}\label{prop:WIT1 is semistable}
Let $\cE$ be a sheaf on $X$ with Hilbert polynomial
$P_{\cE}(s)=dhs+d$ and $d>0.$ Then $\cE$ is WIT$_1$-$\Phi$ if and
only if $\cE$ is a torsion free $\mu_S$-semistable sheaf.
Analogously if $P_{\cE}(s)=dhs-d$ with $d>0,$ then $\cE$ is
WIT$_1$-$\what\Phi$ if and only if $\cE$ is a torsion free
$\mu_S$-semistable sheaf.
\end{lem}
\begin{proof}
If $\cE$ is WIT$_1$-$\Phi$, then any subsheaf is  WIT$_1$-$\Phi$ as
well. Since by Remark \ref{r:propert} torsion sheaves are
WIT$_0$-$\Phi$, this proves that $\cE$ is torsion-free. Moreover, if
$\cF\hookrightarrow\cE$ is a subsheaf, by Corollary \ref{c:Wit},
$\mu_S(\cF)\leq 1/h=\mu_S(\cE),$ so that $\cE$ is
$\mu_S$-semistable. The converse is part of Proposition
\ref{prop:WIT classification}. The proof for $\what\Phi$ is similar.
\end{proof}

\begin{prop} \label{p:transform}
Let $\cT$ a non-zero torsion sheaf on $X$. Then the transform
$\what\cT$ with respect to both equivalences $\Phi$ and $\what\Phi$
is a torsion-free $\mu_S$-semistable sheaf.
\end{prop}
\begin{proof}
Since $\cT$ is a torsion sheaf, $\cT$ is WIT$_0$-$\Phi$ and its
Hilbert polynomial is $P_{\cT}(s)=d$ with $d>0$. The transform
$\what\cT$ is WIT$_1$-$\what\Phi$ and, by Proposition \ref{prop:RR},
its Hilbert polynomial is $P_{\what\cT}(s)=dhs-d$. We conclude by
Lemma \ref{prop:WIT1 is semistable}. The proof for $\what\Phi$ is
the same.
\end{proof}

We state now the result that ensures the preservation of
semistability under the Fourier-Mukai transforms $\Phi$ and
$\what\Phi$.

\begin{thm} \label{t:preservation}Let $X$ be a projective Gorenstein curve of arithmetic
genus one and trivial dualizing sheaf. Fix a polarization $H$ on
$X$. Let $\cE$ be a pure dimensional sheaf on $X$. If $\cE$ is
$\mu_S$-semistable with respect to $H$, then its transform
$\what\cE$ with respect to both equivalences $\Phi$ and $\what\Phi$
is also  $\mu_S$-semistable with respect to $H$.
\end{thm}
\begin{proof} Let $P_{\cE}(s)=rs+d$ be the Hilbert polynomial of
$\cE$ with respect to the fixed polarization $H$ and, as before,
denote by $h$ the degree of $H$. If $r=0$, the result is proved in
Proposition \ref{p:transform}, so that we can assume that $\cE$ is a
torsion-free sheaf. Let us distinguish the following cases:

\medskip
(1)  If $dh=r$, $\what\cE$ is a torsion sheaf by Proposition \ref{prop:RR} and thus semistable.

\medskip
(2) Suppose now that $dh>r$. By Proposition \ref{prop:WIT
classification}, $\cE$ is WIT$_0$-$\Phi$, so that $\what\cE$ is
WIT$_1$-$\what\Phi$. The same argument as in Lemma \ref{prop:WIT1
is semistable} shows that $\what{\cE}$ is torsion-free. If
$\what{\cE}$ is not $\mu_S$-semistable, there is an exact sequence
$$0\longrightarrow\cF\longrightarrow\what\cE\longrightarrow\cG\longrightarrow
0\,,$$ with $\mu_S(\cG)<\mu_S(\what\cE)<\mu_S(\cF)$. Moreover, the
existence of Harder-Narasimhan filtrations allows us to assume that
$\cG$ is a torsion-free $\mu_S$-semistable sheaf. Since $\what{\cE}$
is torsion-free and WIT$_1$-$\what\Phi$, by Proposition
\ref{prop:WIT classification} $\mu_S(\what\cE)<-1/h$ where the last
inequality is strict because $\cE$ is torsion-free. Then $\cG$ is
a $\mu_S$-semistable sheaf with $\mu_S(\cG)<-1/h$, so that $\cG$ is
WIT$_1$-$\what\Phi$ by Proposition \ref{prop:WIT
classification2}. By applying the inverse Fourier-Mukai transform
one obtains that $\cF$ is WIT$_1$-$\what\Phi$ as well and that there
is an exact sequence
$$0\longrightarrow\what\cF\longrightarrow\cE\longrightarrow\what\cG\longrightarrow 0\,.$$
Since $\cE$ is torsion-free, one has $r(\what\cF)\neq 0$ and
$1/h<\mu_S(\what\cF)\leq\mu_S(\cE)$ where the first inequality is
due to the fact that $\what\cF$ is WIT$_0$-$\Phi$ and the second is
by the $\mu_S$-semistability of $\cE.$ Since
$$\mu_S(\what\cE)=\frac{\mu_S(\cE)}{1-h\mu_S(\cE)}\quad\text{ and }\quad \mu_S(\cF)=\frac{\mu_S(\what\cF)}{1-h\mu_S(\what\cF)}\,,$$
one obtains $\mu_S(\cF)\leq\mu_S(\what\cE)$; this contradicts
$\mu_S(\cF)>\mu_S(\what\cE).$

\medskip
(3)  Suppose finally that $dh<r$. Let us prove that $\what\cE$ is a
torsion-free sheaf. Indeed, a torsion subsheaf $\cT\neq 0$ of
$\what\cE$ should necessary be WIT$_0$-$\what\Phi$ and
$\mu_S(\what\cT)=1/h.$ By applying $\what\Phi$ we obtain
$\what\cT\hookrightarrow\cE$ which contradicts the
$\mu_S$-semistability of $\cE$. If $\what\cE$ is not
$\mu_S$-semistable, as in the previous case we can assume the
existence of an exact  sequence
$$0\longrightarrow\cF\longrightarrow\what\cE\longrightarrow\cG\longrightarrow 0$$
with $\cF$ torsion-free, $\cG$ WIT$_0$-$\what\Phi$ and
$\mu_S(\cF)>\mu_S(\what\cE)>\mu_S(\cG).$ Moreover $\what\cE$ is
WIT$_0$-$\what\Phi$, so that $\mu_S(\cF)>\mu_S(\what\cE)>-1/h$. Then
$\cF$ is WIT$_0$-$\what\Phi$, by Proposition \ref{prop:WIT
classification2}. Applying the inverse Fourier-Mukai transform we get the
exact sequence
$$0\longrightarrow\what\cF\longrightarrow\cE\longrightarrow\what\cG\longrightarrow 0\, .$$
Since $\mu_S(\cF)>\mu_S(\what\cE)>-1/h$ we obtain
$$\mu_S(\what\cF)=\frac{\mu_S(\cF)}{1+h\mu_S(\cF)}> \frac{\mu_S(\what\cE)}{1+h\mu_S(\what\cE)}=\mu_S(\cE)\,,$$ but this contradicts the fact that $\cE$ is $\mu_S$-semistable.

An analogous discussion proves the result for the Fourier-Mukai
transform $\what{\Phi}$.
\end{proof}
The proof of the theorem implies the following consequences

\begin{cor} Let $\what{\cE}$ be  the transform of a $\mu_S$-semistable sheaf $\cE$ with respect to $\Phi$ (resp.~$\what{\Phi}$). The following holds \begin{enumerate}\item If $\cE$ is a torsion-free sheaf and $\mu_S(\cE)\neq 1/h$ (resp.~$-1/h$), then $\what{\cE}$ is also torsion-free.
\item If $\mu_S(\cE)\neq 1/h$ (resp.~$-1/h$) and
$\cE$ is $\mu_S$-stable, then $\what\cE$ is $\mu_S$-stable as well.
\item If $\mu_S(\cE)= 1/h$ (resp.~$-1/h$), then $\what{\cE}$ is $\mu_S$-stable if and only if $d=1$.
\end{enumerate}
\end{cor}

\begin{rem} \label{r:equiv} Assume that $\cE$ is semistable with
$\mu_S(\cE)= 1/h$ and $d>1$.  If $\cE$ is indecomposable (for instance, if it is stable), the transform $\what{\cE}$ is a  torsion sheaf and it is indecomposable by Proposition \ref{p:simple}; thus it is supported at a single
point $x\in X$.  When $x$ is a smooth point, then $\what{\cE}\simeq
\cO_{X,x}/\mathfrak{m}_x^d$. When $x$ is singular, the structure of torsion sheaves
supported at $x$ is much more complicated (see
\cite{BuKr04} for more details); one can see, however, by induction
on the length $d$, that $\what{\cE}$  is always $S$-equivalent to
$\oplus_{d}\cO_x$.
\end{rem}

Since the transform $\what{\cO_x}=\Phi(\cO_x)$ is the ideal sheaf
$\mathfrak m_x$ of the point $x$, and $\what\Phi(\cO_x)=\mathfrak
m_x^\ast$, we deduce the following stability result.

\begin{cor} If $X$ is a polarized Gorenstein curve of arithmetic genus one and trivial dualizing sheaf, then the maximal ideal $\mathfrak{m}_x$ for any point $x\in X$ and its dual $\mathfrak{m}_x^\ast$ are stable sheaves.
\end{cor}

\begin{rem} When $X$ is irreducible, this is a trivial fact. For Gorenstein reducible curves of arithmetic genus bigger or equal than 2, the semistability of $\mathfrak m_x$ for an arbitrary point $x$ of $X$ has been recently proved  in \cite{CCE07} using different techniques.
\end{rem}

The equivalences $\Phi$ and $\Psi$ define scheme isomorphisms
between the corresponding moduli spaces:
\begin{cor}\label{c:isom} Let $(r,d)$ be a pair of integers with $r\geq 0$.
\begin{enumerate}\item The Fourier-Mukai functors $\Phi$ and $\Psi$ induce scheme isomorphisms of moduli spaces
\begin{align*}
\cM_X(r,d)& \simeq \cM_X(dh-r,-d)\quad \text{ for $d/r >1/h$,}
\\
\cM_X(r,d)& \simeq \cM_X(-dh+r,d)\quad \text{for $d/r \leq 1/h$, and}
\\
\cM_X(r,d)& \simeq \cM_X(r, d+r)\, .
\end{align*}
\item The moduli space $\cM_X(r,d)$ is isomorphic either to $\cM_X(0,d_0)\simeq Sym^{d_0}(X)$ with $d_0>0$,
or to  $\cM_X(r_0,0)$ with $r_0>0$  or to $\cM_X(r_0,d_0)$ with
$2r_0/h\leq d_0< r_0$.
\end{enumerate}
\end{cor}
\begin{proof} For the first part see for instance \cite[Corollary~2.65]{BBH08}. Let us prove the second one.
By Remark \ref{r:equiv} and arguing as in the proof of \cite[Corollary~3.33]{BBH08}, one proves that $\cM_X(0,d)\simeq Sym^d(X)$. Consider the family $B$ of all pairs of integers $(r', d')$ with
$r' \geq 0$ that are related with $(r,d)$ by the isomorphisms in (1).
Take $(r_0,d_0)$ in $B$ such that $r_0 \ge 0$ is the minimum $r'$ among the pairs in $B$. If $r_0=0$, then $d_0>0$ and $\cM_X(r,d)\simeq \cM_X(0,d_0)$.
Assume then that $r_0>0$. Since by
applying the equivalence $\Psi$ it is possible to increase $d$ and
by applying $\what{\Psi}$ to decrease it, if one considers now  $C$
as the family of all pairs in $B$ with $r'=r_0$, one can choose
$d_0$ as the minimum $d'$ among pairs in $C$ such that $0\leq d_0<r_0$. If $d_0=0$, then $\cM_X(r,d)\equiv \cM_X(r_0,0)$. If $d_0>0$, we claim that all the sheaves in $\cM_X(r_0,d_0)$ are
WIT$_0$-$\Phi$. Indeed, otherwise  $r_0\leq r_0-d_0h$ by the choice
of $r_0$ and then $d_0<0$ which contradicts our choice of $d_0$.
Then, it has to be $r_0\leq d_0h-r_0$ so that $2r_0/h\leq d_0<r_0$
and the proof is complete.
\end{proof}

\begin{rem}\label{r:2com} Notice that if $X$ has only two irreducible components, as it happens for the Kodaira
fibers $E_2$ (cf.~Figure \ref{fig:e2}) and $III$, and the
polarization $H$ has degree $h=2$, Corollary \ref{c:isom} reduces the
study of the moduli spaces $\cM_X(r,d)$ just to the case $d=0$. Some
results in this case can be found in the next section.
\end{rem}

\begin{rem} If $X$ is irreducible, we can take $h=1$, so that $r$ and $d$ are the usual rank and degree. In this situation, the last case in (2) of Corollary \ref{c:isom} does not occur; moreover, there is an isomorphism $\cM(r_0,0)\simeq \cM_X(0,r_0)$. We get then that $\cM(r,d)\simeq Sym^{r_0}(X)$. Using the transforms $\Phi$ and $\Psi$ and the Euclid algorithm, one can see that $r_0=\gcd(r,d)$, as proven in \cite[Chapter 6]{BBH08} by generalizing an argument described for smooth elliptic curves by Bridgeland \cite{Bri98} and Polishchuk \cite{Pol03}. A consequence is that there are no stable sheaves on $X$ if $\gcd(r,d)>1$. As already mentioned in the introduction, a complete description of these moduli spaces can be found in \cite[Chapter 6]{BBH08}.
\end{rem}

\subsection{Relative moduli spaces} Let $p\colon S\to B$ a genus one fibration, that is, a projective Gorenstein morphism whose fibers are curves of arithmetic genus one and trivial dualizing sheaf but without further assumptions on $S$, $B$ or the fibers; in particular, non-reduced and non-irreducible fibers are allowed. Consider the relative integral functor
$$
\Phi=\fmf{\cI_\Delta}{S}{S}\colon \cdbc{S}\to  \cdbc{S}\,,
$$
with kernel the ideal sheaf $ \cI_\Delta$ of the relative diagonal
immersion $\delta\colon S \hookrightarrow S\times_B S$. By \cite[Proposition~2.16]{HLS08}, it is an equivalence of
categories.

Fix a relative polarization $\mathcal{H}$ on the fibers of $p$ and
denote $\hat{p}\colon \cM_{S/B}(r,d)\to B$ the relative coarse
moduli space of $\mu_S$-semistable sheaves on the fibers (with
respect to the induced polarization) that have Hilbert polynomial
$P(s)=rs+d$. Closed points of the fiber
$\hat{p}^{-1}(b)=\cM_{S_b}(r,d)$ represent $S$-equivalence classes
of $\mu_S$-semistable sheaves on the fiber $S_b$ with Hilbert
polynomial $P(s)=rs+d$. Denote also $\Psi$ the relative
auto-equivalence of $\cdbc{S}$ given by twisting by the line bundle
$\cO_{S/B}(\mathcal{H})$.

Taking into account \cite[Corollary~6.3]{BBH08} and Corollary \ref{c:isom}, we get:
\begin{cor}\label{c:isom2}  Let $(r,d)$ be a pair of integers with $r\geq 0$.
\begin{enumerate}\item The Fourier-Mukai functors $\Phi$ and $\Psi$ induce scheme isomorphisms of moduli spaces
\begin{align*}
\cM_{S/B}(r,d)& \simeq \cM_{S/B}(dh-r,-d) \quad \text{for $d/r > 1/h$, }
\\
\cM_{S/B}(r,d)& \simeq \cM_{S/B}(-dh+r,d) \quad \text{for $d/r \leq 1/h$, and}\\
\cM_{S/B}(r,d)&\simeq \cM_{S/B}(r, d+r)\, .
\end{align*}
\item The moduli space $\cM_{S/B}(r,d)$ is isomorphic either to $\cM_{S/B}(0,d_0)\simeq Sym^{d_0}(S/B)$ with $d_0>0$, or to  $\cM_{S/B}(r_0,0)$ with $r_0>0$  or to $\cM_{S/B}(r_0,d_0)$ with $2r_0/h\leq d_0< r_0$.
\end{enumerate}

\end{cor}

\section{Moduli spaces of degree zero sheaves for $E_N$} In this section, we give a description of
the connected component of the moduli space $\cM_X(r,0)$ with $r>0$
containing vector bundles when $X$ is a curve of type $E_N$,  that
is, a cycle of $N$ projective lines (cf.~Figure \ref{fig:en}). The
description is achieved by combining two different ingredients; the
first one is the description of indecomposable torsion-free sheaves
on cycles $E_N$ given in \cite{DG01,BBDG07}, and the second one is
the description of (semi)stable line bundles on tree-like curves and
cycles carried out by one of the authors in \cite{LM05a,LM05}.

\begin{figure}[htbp] 
   \centering
   \includegraphics[width=1.5in]{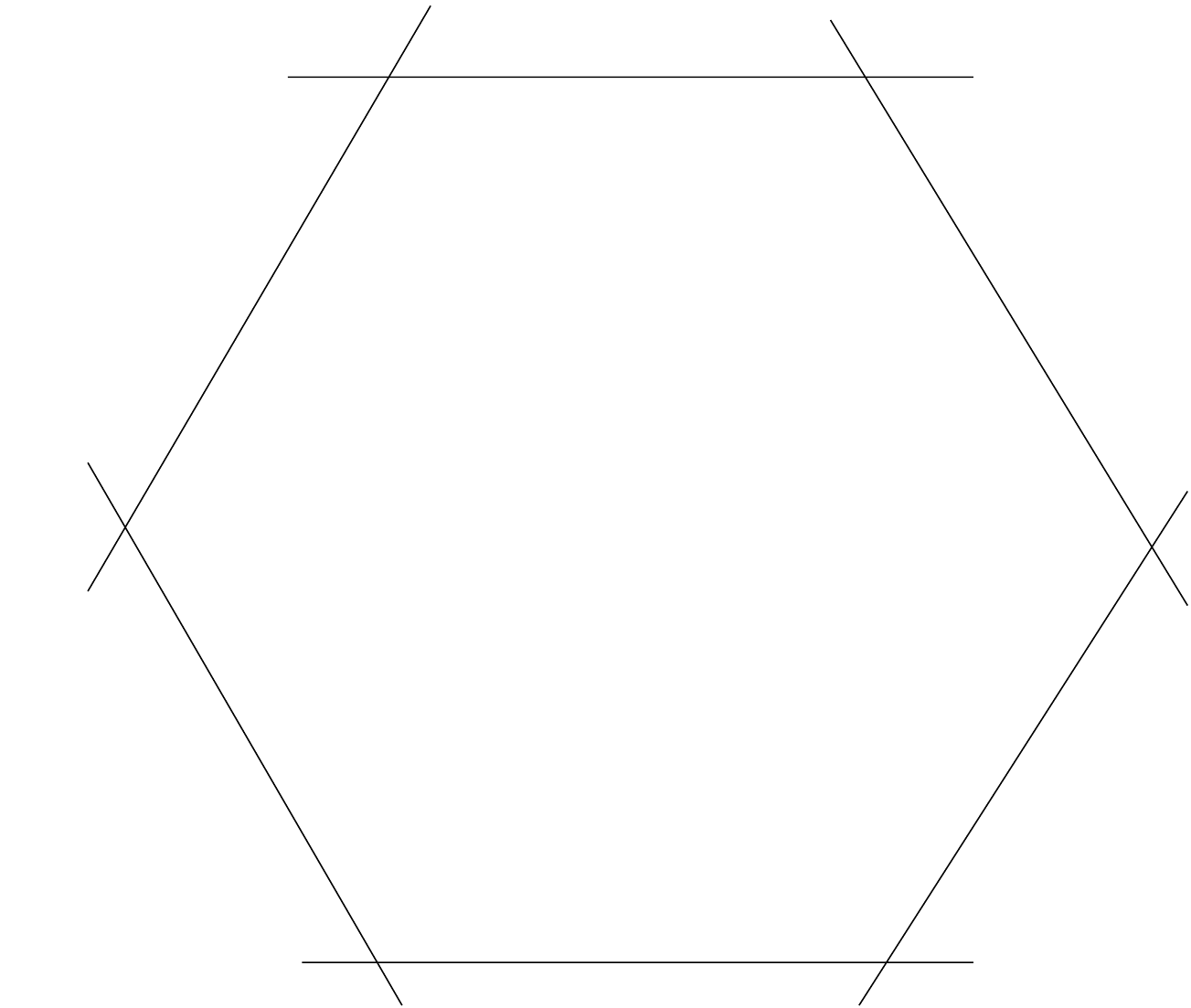}
   \caption{The curve $E_6$}
   \label{fig:en}
\end{figure}

\subsection{Coherent sheaves on reducible curves}
We collect some results about coherent sheaves on reducible curves.
Let $X$ be any projective connected and reduced curve over an
algebraically closed field $k$. Denote by  $C_1,\dots,C_N$ the
irreducible components of $X$ and by $x_1,\dots,x_k$ the
intersection points of $C_1,\dots,C_N$. Let $\cE$ a coherent sheaf
on $X$ and denote
$\cE_{C_i}=(\cE\otimes\mathcal{O}_{C_i})/\text{torsion}$ its
restriction to $C_i$ modulo torsion. Let $r_i=r_i(\cE)$ and $d_i=d_i(\cE)$ be the rank
and the degree of $\cE_{C_i}$.
\begin{defin} The  \emph{multirank}  and  \emph{multidegree} of a coherent sheaf $\cE$ on $X$ are the
$N$-tuples $\underline{r}(\cE)=(r_1,\dots, r_N)$ and
$\underline{d}(\cE)=(d_1,\dots, d_N)$.
\end{defin}
Let $H$ be a polarization on $X$ of degree $h$ and denote by $h_i$
the degree of $H$ on $C_i$. As in \cite{Ses82}, for any pure
dimension one sheaf $\cE$ on $X$, there is an exact sequence
\begin{equation}\label{exactseq} 0\to \cE \to \cE_{C_1}\oplus \dots
\oplus \cE_{C_N}\to T \to 0\,,
\end{equation} where $T$ is a torsion sheaf whose
support is contained in the set $\{ x_1,\dots,x_k\}$. The first arrow in the sequence \eqref{exactseq} is the composition of the canonical morphisms $\cE \to \pi_\ast(\pi^\ast\cE) \to  \pi_\ast(\pi^\ast\cE/tors)$, where $\pi$ is the normalization morphism of $X$.

Let $P_{\cE}(s)=rs+d$ be the Hilbert polynomial of $\cE$ with
respect to $H$. Since the Hilbert polynomial is additive, from the
above exact sequence, one obtains that
$$r=r_1h_1+\dots+r_Nh_N\, .$$ Hence, there is a natural decomposition
\begin{equation}\label{e:des}\cM_X(r,d)=\coprod_{\underline{r} \in \mathbb{Z}^N}
\cM_X(\underline{r},d)
\end{equation}where $\cM_X(\underline{r},d)$ is the moduli space of those semistable
sheaves $\cE$ on $X$ of multirank $\underline{r}$ and the union runs
over all  $N$-tuples $\underline{r}=(r_1,\dots,r_N)$ of
non-negative integers such that $r=r_1h_1+\dots+r_Nh_N$.

This decomposition becomes necessary in the analysis of some moduli
spaces $\cM_X(r,d)$ because, as it was mentioned in Remark
\ref{r:linearity}, if $\Gamma$ is an equivalence of $\cdbc{X}$ and
$\cE$ is a WIT-$\Gamma$ sheaf, it is not true in general that the
coefficients of the Hilbert polynomial of $\Gamma(\cE)$ are linear
functions of the coefficients of the Hilbert polynomial of $\cE$.
This means that the action of an arbitrary equivalence of $\cdbc{X}$
does not send all connected components of the moduli space
$\cM_X(r,d)$ into the same moduli space $\cM_X(r',d')$. Here we have
an example of this fact.

\begin{exe}\label{e:nonlinear} Let $X$ be a curve of type $E_2$, that is,
two rational curves $C_1$ and $C_2$ meeting transversally at two
points with a polarization $H$ such that $h_1=h_2=1$ (Figure \ref{fig:e2}). Let $\calL$ be
a line bundle on $X$ with multidegree
$\underline{d}(\calL)=(d_1,d_2)$ with $d_1\neq d_2$ and take
$\Gamma=\fmf{\delta_\ast \calL}{X}{X}$ the equivalence defined by
twisting by $\calL$. If $\cE$ is a coherent sheaf on $X$ with
Hilbert polynomial $P_{\cE}(s)=rs+d$ and multirank
$\underline{r}(\cE)=(r_1,r_2)$, the Hilbert polynomial of
$\Gamma(\cE)$ is equal to $rs+d+(r_1d_1+r_2d_2)$. Thus, the
connected component $\cM_X((r,0),d)\subseteq \cM_X(r,d)$ is sent by
$\Gamma$ into the moduli space $\cM_X(r, d+rd_1)$ while the
component $\cM_X((0,r),d)$ is mapped into $\cM(r, d+rd_2)$.
\end{exe}
The following proposition provides better invariants. If
$K(\cdbc{X})$ denotes the Grothendieck group of the triangulated
category $\cdbc{X}$ (cf.~\cite{Gro77}), one has $K(\cdbc{X})\simeq
K(\text{Coh}(X))$ and this group is usually denoted $K_\bullet(X)$.

\begin{prop} \label{p:grupoK} Let $X$ be any reduced connected and projective curve. If every irreducible
component of $X$ is isomorphic to $\mathbb{P}^1$ the projective
line, then there is an isomorphism $K_\bullet(X)\simeq
\mathbb{Z}^{N+1}$ where $N$ is the number of irreducible components
of $X$.
Moreover the above isomorphism is defined by sending the class of
any coherent sheaf $[\cE]$ in $K_\bullet(X)$ to
$(\underline{r}(\cE), \chi(\cE))\in \mathbb{Z}^{N+1}$.
\end{prop}
\begin{proof} Let us denote by $C_1,\dots, C_N$ the irreducible components of the curve
$X$. Since the curve is connected and its irreducible components are
isomorphic to $\mathbb{P}^1$ we conclude that any two points are
rationally equivalent, that is $A_0(X)=\bbZ[x]$, where $[x]$ is the
class of a point of $X$. On the other hand it is well known
\cite[Example 1.3.2]{Ful} that the $n$-th Chow group of an
$n$-dimensional scheme is the free abelian group on its
$n$-dimensional irreducible components, therefore
$A_1=\bbZ[C_1]\oplus\dots \oplus\bbZ[C_N]$.

The normalization $\pi\colon \widetilde X \to X$ is a Chow envelope
of $X$, thus $\pi_*\colon K_\bullet(\widetilde X) \to K_\bullet(X)$
is surjective \cite[Lemma~18.3]{Ful}. Since $\widetilde X=\widetilde
C_1\coprod\dots\coprod\widetilde C_N$ one has
$$K_\bullet(\widetilde X)\simeq K_\bullet(\widetilde
C_1)\oplus\dots\oplus K_\bullet(\widetilde C_N).$$ Taking into
account that $\widetilde C_i\simeq\bbP^1$ we get
$K_\bullet(\widetilde X)=\bigoplus_{i=1}^N\bbZ[\cO_{\widetilde
C_i}]\oplus\bbZ[\cO_{\tilde x_i}]$, where $\tilde x_i\in\widetilde
C_i$. Given any two points $\tilde x_i$, $\tilde y_i$ in $\widetilde
C_i\simeq\bbP^1$ we know that $[\cO_{\tilde x_i}]=[\cO_{\tilde
y_i}]$. Therefore, since $X$ is connected, the surjectivity of
$\pi_*$ implies that $K_\bullet(X)$ is generated by
$$[\cO_{C_1}],\dots,[\cO_{C_N}], [\cO_{x}].$$
By the Riemann-Roch theorem for algebraic schemes \cite[Theorem
18.3]{Ful} there is a homomorphism $\tau_X\colon K_\bullet(X)\to
A_\bullet(X)\otimes_\bbZ\bbQ$ with the following properties:

\begin{enumerate}
\item  For any $k$-dimensional closed subvariety $Y$ of $X$, one has
$\tau_X([\cO_Y])=[Y]+ \mathrm{terms\ of\ dimension} < k$,
\item $(\tau_X)_\bbQ\colon
K_\bullet(X)\otimes_\bbZ\bbQ\xrightarrow{\sim}
A_\bullet(X)\otimes_\bbZ\bbQ$ is an isomorphism.
\end{enumerate}
Using the first
property one easily sees that
$\{\tau_X([\cO_{C_1}]),\dots,
\tau_X([\cO_{C_N}]),\tau_X([\cO_{x}])\}$ is a basis of
$A_\bullet(X)\otimes_\bbZ\bbQ$. Taking into account that
$\{[\cO_{C_1}],\dots,[\cO_{C_N}], [\cO_{x}]\}$ is a system of
generators of $K_\bullet(X)$, the second property of $\tau_X$
implies that it is also a basis.

The final statement follows now straightforwardly by applying the
integer valued mapping $(\underline{r}(-), \chi(-))$ to the basis of
$K_\bullet(X)$ above constructed.
\end{proof}

All the functions defining  the isomorphism of Proposition \ref{p:grupoK}, that is, the ranks
$r_i$ and the Euler characteristic
$\chi$, are additive on exact triangles of $\cdbc{X}$. Hence, any
equivalence  $\Gamma$ of $\cdbc{X}$ induces a group automorphism
$\gamma$ of $K_\bullet(X)$, such that there is a commutative square
$$
   \xymatrix{
       \cdbc{X}\ar[r]^{\Gamma} \ar[d] & \cdbc{X} \ar[d]\\
       K_\bullet(X)\ar[r]^{\gamma} & K_\bullet(X)
       }
$$
where the vertical arrows are the natural ones.

Note that if $\cE$ is a vector bundle on $X$, then $\sum_{i=1}^N
d_i=d$. Thus, the category of vector bundles on $X$ of rank $r$ and
degree $d$ (whose Hilbert polynomial is $P_{\cE}(s)=rhs+d$)
decomposes as
$$\text{VB}_X(r, d)\simeq \coprod_{\underline{d}\in\mathbb{Z}^N}\text{VB}_X(r,\underline{d})\,,$$
where now the union runs over all $\underline{d}=(d_1,\dots,d_N)\in
\mathbb{Z}^N$ such that $\sum_{i=1}^N d_i=d$. However, since for
non-locally free sheaves it is not true that $\sum_{i=1}^N d_i=d$,
there is not a similar  decomposition for the moduli space
$\cM_X(r,d)$.

To finish this subsection, let us show how $\mu_S$-semistability
behaves under direct and inverse images by Galois coverings of
reducible curves. For non-singular projective and irreducible
varieties, similar results were proved by Takemoto in \cite{Tak73}.

\begin{lem}\label{l:DirectInverse} Let $X$ be a projective connected and reduced curve whose irreducible components are smooth. Let $H$ be a polarization on $X$ and let $f\colon Y\to X$ be an \`{e}tale Galois covering of degree $n$ where  $Y$ is also connected.
\begin{enumerate} \item If $\mathcal{E}$ is a torsion-free sheaf on $X$ such that $f^\ast \cE$ is $\mu_S$-semistable with respect to $f^\ast H$, then $\mathcal{E}$ is $\mu_S$-semistable with respect to $H$.
\item A torsion-free sheaf $\cF$ on $Y$ is  $\mu_S$-semistable with respect to  $f^\ast H$ if and only if $f_\ast(\cF)$ is $\mu_S$-semistable with respect to  $H$.
\end{enumerate}
\end{lem}
\begin{proof} (1) By \eqref{exactseq}, if $C_1,\dots, C_N$ are the irreducible components of $X$, there is an exact sequence $$0\to \cE\to \cE_{C_1}\oplus \dots  \oplus\cE_{C_N}\to  T\to 0\, ,$$ where $\mathcal{E}_{C_i}$ is a vector bundle on $C_i$ for all $i$ and $T$ is a torsion sheaf.
Because $C_i$ is smooth, one has that
$P_{f^\ast\mathcal{E}_{C_i}}(s)=nP_{\mathcal{E}_{C_i}}(s)$.
Moreover, since $f$ is flat,  $$0\to f^\ast \cE\to f^\ast
\cE_{C_1}\oplus \dots \oplus f^\ast \cE_{C_r}\to f^\ast T\to 0$$ is also
an exact sequence. Then the additivity of the Hilbert polynomial
allow us to  conclude that
$P_{f^\ast\mathcal{E}}(s)=nP_{\mathcal{E}}(s)$ where both
polynomials are computed with the corresponding polarizations. Hence
\begin{equation}\label{e:muinversa}
\mu_S(f^\ast \cE)=\mu_S(\cE)
\end{equation} and the result follows.

(2) By (1), to prove the direct implication it suffices to prove
that $f^\ast f_\ast(\cF)$ is semistable with respect to $f^\ast H$. Since any
Galois \`{e}tale covering with $X$ and $Y$ connected is trivialized by
itself, $f^\ast f_\ast(\cF)=\oplus\rho^\ast \cF$ where $\rho$ runs
over the Galois group of $f$. Then the result follows because any
extension of two $\mu_S$-semistable sheaves with the same slope is
$\mu_S$-semistable as well. The converse is true for any finite
morphism.
\end{proof}

It is important to remark that Equation \eqref{e:muinversa} is
true because we are considering the Simpson's slope and not the usual slope, defined for a torsion-free sheaf as the
quotient obtained by dividing the degree by the rank. For smooth
varieties, it is a well-known fact (cf.~\cite{Tak73})
that if $f$ is an unramified covering then the  relation between the usual slopes of $\cE$ and of
$f^\ast \cE$ is given instead by $\mu(f^\ast \cE)=\deg f\cdot\mu(\cE)$.

\subsection{Indecomposable torsion free sheaves on $E_N$} The purpose of this subsection is just to
state some known results about the classification of indecomposable
vector bundles and torsion free sheaves on cycles $E_N$ of
projective lines that we shall use later. This classification was
obtained for the first time by Drozd and Greuel in \cite{DG01} for
arbitrary base fields. Nevertheless, for algebraically closed base
fields, one can find a geometric description of indecomposable
torsion free sheaves on $E_N$ in \cite{BBDG07} that follows the
classical description of vector bundles on elliptic curves given by
Oda in \cite{O71} (see Theorem \ref{t:indecom} below) and allows to
study which of these sheaves are semistable.

Following the same argument that Atiyah used for smooth elliptic
curves and taking into account that $\Ext^1(\cO_{E_N},\cO_{E_N})=k$,
it is possible to inductively prove  the following result.
\begin{lem} \label{l:Atiyah} Let $E_N$ be a cycle of $N$ projective lines. For any integer $m\geq 1$ there is a unique indecomposable vector bundle $\cF_m$
on $E_N$ appearing in the exact sequence $$0\to \cF_{m-1}\to
\cF_m\to \cO_{E_N}\to 0\, ,\quad \cF_1=\cO_{E_N}\,.
$$
\end{lem}

\begin{thm}\label{t:indecom}\emph{\cite[Theorem 19]{BBDG07}} Let $E_N$ be a cycle of $N$ projective lines
and $I_k$ be a chain of $k$ projective lines (cf.~Figure
\ref{fig:ik}). Let $\cE$ be an indecomposable torsion free sheaf on
$E_N$. The following holds:
\begin{enumerate}
\item If $\cE$ is a
vector bundle, there is an \'{e}tale covering $\pi_r\colon E_{rN}\to
E_N$, a line bundle $\calL$ on $E_{rN}$  and a number $m\in \mathbb{N}$ such
that $$\cE\simeq \pi_{r\ast}(\calL\otimes\cF_m)\, .$$
The integers $r$, $m$ are determined by $\cE$. Moreover, when $r>1$ the multidegree
$\underline{d}(\calL)$ of the line bundle $\calL$ is non-periodic.
\item If $\cE$ is not locally free, then there exits a finite map $p_k\colon I_k\to E_N$ (defined as the composition of some  $\pi_r$ and some closed immersion $i\colon I_k\hookrightarrow E_{rN}$) and a line bundle $\calL$ on $I_k$ such that $\cE\simeq p_{k\ast}(\calL)$.
   \end{enumerate}

\end{thm} Remember that if $\underline{d}=(d_1,\dots, d_N,
d_{N+1},\dots, d_{2N},\dots,d_{(r-1)N},\dots,d_{rN})$, then the
non-periodicity means that $\underline{d}\neq \underline{d}[t]$ for
$t=1,\dots, r-1$ where $$\underline{d}[1]=(d_{N+1},\dots,
d_{2N},\dots, d_{(r-1)N},\dots,d_{rN},d_1,\dots, d_N)$$ and
$\underline{d}[t]=(\underline{d}[t-1])[1]$.

\begin{figure}[h] 
  \centering
  \includegraphics[width=1in]{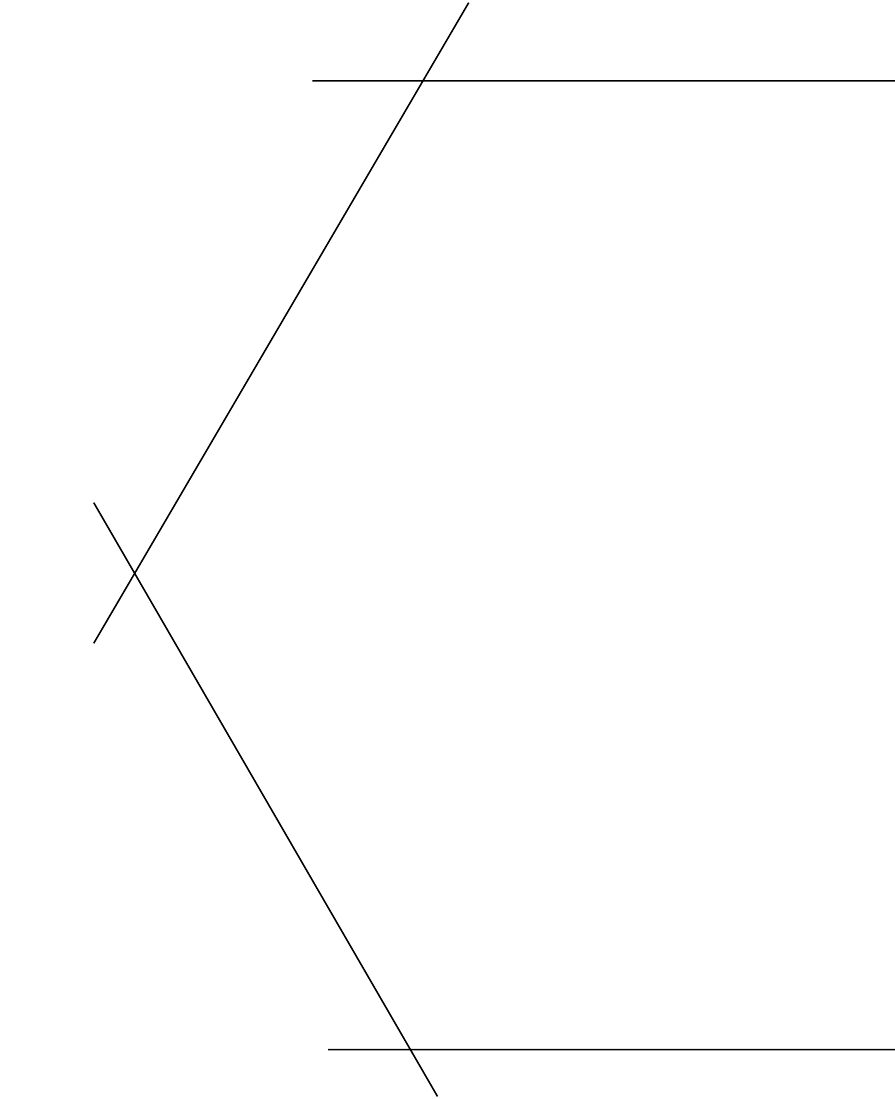}
  \caption{The curve $I_4$}
  \label{fig:ik}
\end{figure}

\subsection{Locally free sheaves on cycles}
\begin{prop}\label{prop:locally free}
Let $\cE$ be a pure dimension one sheaf on $E_N$. Then
$$\sum_{i=1}^N d_i(\cE)-\chi(\cE)\leq 0\,,
$$
and $\sum_{i=1}^N d_i(\cE)-\chi(\cE)= 0$ if and only if $\cE$ is locally free.
\end{prop}
\begin{proof}
Since the function $\sum_{i=1}^N d_i(\cE)-\chi(\cE)$ is additive over direct sums of sheaves, we can assume that $\cE$ is indecomposable. We can then apply Theorem \ref{t:indecom}. If $\cE$ is not locally free, then $\cE \simeq p_{k\ast}(\calL)$, where $\calL$ is a line bundle on a chain $I_k$ of projective lines and $p_k\colon I_k \to E_N$ is a finite morphism. One has
$\chi(\cE)=\chi(\calL)= 1+ \sum_{D} d(\calL_{D})$,
where the sum runs over the irreducible components $D$ of $I_k$. Moreover $\sum_{D} d(\calL_{D})= \sum_{i=1}^N d_i(\cE)$ and then
$\sum_{i=1}^N d_i(\cE)-\chi(\cE)<0$.

If $\cE$ is locally free, then $\cE \simeq \pi_{s\ast}(\calL \otimes \cF_m)$ for a finite morphism $\pi_s\colon E_{sN}\to E_N$ and a line bundle $\calL$ on $E_{sN}$. Thus
$\chi(\cE)=\chi(\calL\otimes \cF_m)=m \chi(\calL)= m\sum_{j=1}^{sN} d(\calL_{D_j})$,
where the sum runs over the irreducible components $D_j$ of $E_{sN}$. Thus
$\sum_{i=1}^N d_i(\cE)-\chi(\cE)= 0$.
\end{proof}

\begin{rem}
Proposition \ref{prop:locally free} is not true for the chain $I_N$. For any pure
dimension one sheaf $\cE$ on $I_N$ one has $\sum_{i=1}^N
d_i-\chi(\cE) < 0$.
\end{rem}

\subsection{Stable sheaves of degree zero on $E_N$}
Let $X=E_N$ be as above, a cycle of projective lines. Suppose now
that the number of irreducible components of $X$ is $N\geq 2$. Fix a
polarization $H$ on $X$ of degree $h$ and let $\cE$ be a coherent
sheaf on $X$. Since $\chi(\cO_X)=0$, the second coefficient of the
Hilbert polynomial, that coincides with its Euler characteristic
$\chi(\cE)$, can be seen in view of Remark \ref{r:RR} as its degree.
Hence, from now on, if the Hilbert polynomial of $\cE$ is
$P_{\cE}(s)=rs$, then we will refer to it as a degree 0 sheaf.

When $r=h$, the structure of the component $\cM_X((1,\dots,
1),0)\subset \cM_X(h,0)$ was determined by one of the authors in
\cite{LM05} and \cite{LM05a}. There, one can find the  following
result.
\begin{lem} \label{l:rank1} Let $\cE$ be a pure dimension one sheaf on $X$ of degree
0 and multirank $\underline{r}(\cE)=(1,\dots, 1)$.
\begin{enumerate}
\item The (semi)stability of $\cE$ does not depend on the polarization.
\item $\cE$ is $\mu_S$-stable if and only if it is a line bundle and its multidegree is $\underline{d}(\cE)=(0,\dots,0)$.
\item If $\cE$ has multirank $\underline{r}(\cE)=(1,\dots, 1)$ and is strictly semistable, then its graded object is $Gr(\cE)\simeq \oplus_{i=1}^N \cO_{C_i}(-1)$, where $C_1,\dots,C_N$ are the irreducible components of $X$.
\end{enumerate}
\end{lem}

\begin{lem} \label{l:tensor} Let $\calL$ be a line bundle on $X$. For any integer $m\geq 1$, the vector bundle $\cF_m\otimes \calL$ is $\mu_S$-semistable if and only if $\calL$ is $\mu_S$-semistable. For $m>1$, $\cF_m\otimes \calL$ is never a $\mu_S$-stable sheaf.
\end{lem}

\begin{proof} Using the exact sequences that define $\cF_m$ in Lemma \ref{l:Atiyah} one gets that $P_{\cF_m\otimes \calL}(s)=mP_{\calL}(s)$ so that $\mu_S(\cF_m\otimes \calL)=\mu_S(\calL)$ for any $m\geq 1$. Thus, the statement about the non-stability $\cF_m\otimes \calL$ for $m>1$ is straightforward.

Assume that $\cF_m\otimes \calL$ is $\mu_S$-semistable. If
$\calL$ is not $\mu_S$-semistable, by Lemma 3.3 in
\cite{LM05}, there exists a proper subcurve $Z\subset X$ such that
$\mu_S(\calL)>\mu_S(\calL_Z)$ where $\calL_Z$ is
the restriction of $\calL$ to $Z$. Since $P_{\cF_m\otimes
\calL_Z}(s)=mP_{\calL_Z}(s)$, the quotient $\cF_m\otimes
\calL_Z$ contradicts the semistability of $\cF_m\otimes
\calL$.  The converse is proved by induction on $m$ taking
into account that the category of semistable sheaves of fixed slope
is closed under extensions.
\end{proof}

In analogy to what happens for smooth elliptic curves and
irreducible  projective curves of arithmetic genus one, the
following theorem proves that there are no stable degree 0 sheaves
of higher rank.
\begin{thm} \label{t:stables}Let $X$ be a polarized curve of type $E_N$ with $N\geq 2$. Let $\cE$ be a sheaf of pure dimension 1 on $X$ with Hilbert polynomial
$P_{\cE}(s)=rs$ and multirank $\underline{r}(\cE)=(r_1,\dots,
r_N)$. If $\cE$ is $\mu_S$-stable, then either it is isomorphic to
$\cO_{C_i}(-1)$ for some $i=1,\dots, N$  or it is a locally free
sheaf with multirank $\underline{r}(\cE)=(1,\dots,1)$ and
multidegree $\underline{d}=(0,\dots,0)$.
\end{thm}

\begin{proof}
Assume that $\cE$ is not isomorphic to any of the sheaves $\cO_{C_i}(-1)$,
$i=1,\dots, N$. By \eqref{exactseq}, we have an exact sequence
$$0\to \cE \to \cE_{C_1}\oplus \dots
\oplus \cE_{C_N}\to T \to 0 $$  where $C_1,\dots, C_N$ denote the
irreducible components of $X$. Let us consider a component $C_i$ such that $\cE_{C_i}\neq 0$. Then $\cE_{C_i}$ is a
vector bundle on $C_i\simeq \mathbb{P}^1$ of rank $r_i$ and degree
$d_i$, and
the Grothendieck description of vector bundles on the projective
line gives an isomorphism $$\cE_{C_i}\simeq
\oplus_{j=1}^{r_i}\cO_{\mathbb{P}^1}(\alpha_{i,j})$$ where the
integers $\alpha_{i,j}$ satisfy $\sum_{j=1}^{r_i}\alpha_{i,j}=d_i$. Moreover, since we are assuming that $\cE$ is not isomorphic to $\cO_{C_i}(-1)$, the sheaf
$\cO_{\mathbb{P}^1}(\alpha_{i,j})$ is a strict quotient of $\cE$  for every $j=1, \dots, r_i$, and the
stability of $\cE$ imposes that $\alpha_{i,j}\geq 0$ so that $d_i\geq 0$. By Proposition \ref{prop:locally free},
$\cE$ is a locally free sheaf.

It remains to show that the rank $\bar r=r_1=\dots=r_N$ of the vector bundle $\cE$ is one. Assume that $\bar r>1$. Since any stable
vector bundle is indecomposable, we can apply Theorem \ref{t:indecom}. Notice that we can exclude the case $\cE=\calL\otimes \cF_{\bar r}$ because the sheaves $\calL\otimes \cF_{\bar r}$ are strictly semistable by Lemma \ref{l:tensor}. Then, there
is an \'{e}tale covering $\pi_s\colon E_{sN}\to E_N$, a line bundle
$\calL$ on $E_{sN}$ whose multidegree $\underline{d}(\calL)$ is
non-periodic,  and a number $m\in \mathbb{N}$, such that $$\cE\simeq
\pi_{s\ast}(\calL\otimes\cF_m)\, .$$
 If $m>1$, the sheaf
$\pi_{s\ast}(\calL\otimes\cF_m)$ is not stable because its subsheaf
$\pi_{s\ast}(\calL\otimes \cF_{m-1})$ has the same slope. Hence, $m=1$
and  $\cE\simeq \pi_{s\ast}(\calL)$. Since $\pi_s$ is a finite
morphism, $\chi(\calL)=\chi(\cE)=0$, and the stability of $\cE$
implies that $\calL$ has to be a stable line bundle. By Lemma
\ref{l:rank1}, the multidegree of $\calL$ is
$\underline{d}=(0,\dots, 0)$ which contradicts the non-periodicity.
Then, $\cE$ is a line bundle and we conclude the proof.
\end{proof}

This is related to the following result due to  L. Bodnarchuck, presented by her in VBAC-2007: If  $\cE$ is a simple vector bundle of rank $r$, multidegree $\underline{d}=(d_1,\dots, d_N)$ and degree $d$ on $E_N$, one has $(r, d)=(r, d_1, \dots, d_N)=1$. Moreover, if these conditions are satisfied, the determinant gives an equivalence between the category of simple vector bundles of rank $r$ and multidegree $(d_1,\dots, d_N)$ on $E_N$ and $\Pic^{\underline{d}}(E_N)$. See \cite[Thm.~1.2.2, Rem.~1.2.3]{Bodn08} where the result is proved for Kodaira curves of type II, III and IV and stated for the cycles $E_N$, $N\leq 3$.

Using the results on the moduli space $\cM_X^s((1,\dots,1),0)$
given in \cite[Theorem 4.1]{LM05a} for $r=h$, we can summarize the
structure of the open set of stable degree 0 sheaves as follows.

\begin{cor} Let $X$ be a curve of type $E_N$ with $N\geq 2$ and $H$ a polarization on it of degree $h$. Let $\cM^s_X(r,0)$  be the open subset of stable sheaves with Hilbert polynomial $P(s)=rs$. The following holds:
\begin{enumerate} \item If $r=h$, all the components of $\cM^s_X(r,0)$ given by $\eqref{e:des}$ are empty except $\cM_X^s((1,\dots,1),0)$ which is isomorphic to the multiplicative group $\boldsymbol\mu =k^\ast$. Moreover, the compactification of the component $\cM_X^s((1,\dots,1),0)$ is isomorphic to a rational curve with one node.
    \item If $r=h_i$ for some $i=1,\dots, N$, then  $\cM^s_X(r,0)$ is a single point.
    \item Otherwise, $\cM^s_X(r,0)$ is empty.  \end{enumerate}
\end{cor}

\begin{cor} \label{c:JHfactors}Let $\cE$ be a semistable sheaf on $X$ with Hilbert polynomial $P(\cE)=rs$. If $\cF$ is a Jordan-H\"{o}lder factor of $\cE$, then $\cF$ is isomorphic either to one of the sheaves $\cO_{C_i}(-1)$, where $C_i$ are the  irreducible components of $X$, or to a line bundle $\calL$ on $X$ of multidegree $\underline{d}(\calL)=(0,\dots,0)$.
\end{cor}

\begin{cor} Let $\underline{r}=(r_1,\dots,r_N)\in \mathbb{Z}^N$ such that  $r=r_1h_1+\dots+r_Nh_N$. The dimension of the connected component $\cM((r_1,\dots,r_N),0)$ is equal to the minimum of the $r_i$ with $i=1,\dots, N$.
\end{cor}

\begin{proof}  Suppose that $r_1$ is the minimum of the $r_i$'s. Any sheaf of the form
\begin{equation}\label{e:grad}
(\oplus_{i=1}^u\calL_i)\oplus (\oplus_{j=1}^{N}\mathcal{O}_{C_j}(-1)^{\oplus v_j})\,,
\end{equation}
where $u$ is at most $r_1$, $\calL_i$ are line bundles on $X$ all of them of multidegree $(0,\dots, 0)$ and $u+v_j=r_j$ for all $j$, is a semistable sheaf that defines a point in $\cM((r_1,\dots,r_N),0)$. By Corollary \ref{c:JHfactors}, if $\cE$ is   a semistable sheaf of degree 0 and multirank $\underline{r}(\cE)=(r_1,\dots, r_N)$, its graded object is of the form given by Equation \ref{e:grad}. Since the group of stable line bundles $\Pic_s^0(X)$ of degree 0 on $X$ is determined by the  exact sequence
$$
1\to \kappa^\ast\to \Pic_s^0(X)\to \prod_{i=1}^N \Pic^0(C_i)\to 1\, ,
$$ one gets that the dimension of this component is equal to $r_1$.
\end{proof}

In the case of a rational curve with one node or one  cusp, as
happens also for smooth elliptic curves, it is known \cite{BuKr05}
that all the Jordan-H\"{o}lder factors of any indecomposable sheaf are
isomorphic to each other. This is no longer true for cycles $X=E_N$ of projective lines, as we will now prove.

\begin{lem}\label{lem:Gr Atiyah}
For $m>1$, the graded object of the Atiyah indecomposable vector
bundle $\cF_m$ is $Gr(\cF_m)=\oplus_{i=1}^m
\mathcal{O}_X$
\end{lem}
\begin{proof} This follows from the exact sequences that define $\cF_m$ in Lemma \ref{l:Atiyah} and from the fact that the structural sheaf $\cO_X$ of any cycle $X$ is stable by Lemma \ref{l:rank1}.
\end{proof}

\begin{prop}\label{p:graded} Let $\cE$ be a strictly semistable indecomposable sheaf on $X$ with Hilbert polynomial $P(s)=rs$. If $\cE$ is not locally free, its graded object is $Gr(\cE)\simeq \oplus_{i\in A} \cO_{C_i}(-1)$ for a subset $A\subseteq \{1,\dots,N\}$. If $\cE$ is locally free of rank $r$, then $Gr(\cE)$ is isomorphic either to $\calL^{\oplus r}$ for a line bundle $\calL$ of multidegree $\underline{d}(\calL)=(0,\dots,0)$, or to $ \oplus_{i=1}^N \cO_{C_i}(-1)^{\oplus r}$. \end{prop}

\begin{proof} If $\cE$ is not locally free, by Theorem \ref{t:indecom}, $\cE\simeq p_{k\ast}(\calL)$ where $\calL$ is a line bundle of degree -1 on $I_k$ and $k$, $p_k$ and $\calL$ are determined by $\cE$. Since $p_k=\pi_r\circ i$ for some Galois covering $\pi_r$ and some closed immersion $i$, using Lemma \ref{l:DirectInverse} above and \cite[Lemma 3.2]{LM05}, one has that $\cE$ is $H$-semistable if and only if $\calL$ is $p_k^\ast(H)$-semistable. By \cite[Theorem 4.5]{LM05}, there are not stable line bundles of degree -1 on $I_k$ and there is exactly one strictly semistable line bundle whose graded object is $Gr(\calL)\simeq \oplus_D \mathcal{O}_{D}(-1)$ where the sum runs over all the irreducible components $D$ of $I_k$. Since $p_k$ is a finite morphism,  $p_{k\ast}$ is an exact functor, and then the direct image by $p_k$ of a Jordan-H\"{o}lder filtration for $\calL$ is a filtration for $\cE$. Hence the graded object is also $Gr(\cE)\simeq \oplus_{i\in A} \cO_{C_i}(-1)$ for a subset $A\subseteq \{1,\dots,N\}$.

Suppose now that $\cE$ is a vector bundle.  Assume first that $\cE=\calL\otimes \cF_r$ for a line bundle $\calL$ of multidegree $(0,\dots,0)$. By Lemma \ref{l:tensor}, $\calL$ is semistable;  if it is stable, then $Gr(\cE)=\calL^{\oplus r}$ by Lemma \ref{lem:Gr Atiyah}; if $\calL$ is strictly semistable, then its Jordan-H\"older factors are $\cO_{C_1}(-1),\dots, \cO_{C_N}(-1)$ by Corollary \ref{c:JHfactors}, and again by Lemma \ref{lem:Gr Atiyah}, one has that $Gr(\cE)\simeq \oplus_{i=1}^N \cO_{C_i}(-1)^{\oplus r}$.

By Theorem
\ref{t:indecom}, the only remaining case is when $\cE\simeq \pi_{s\ast}(\calL\otimes\cF_m)$ for an
\'{e}tale covering $\pi_s\colon E_{sN}\to X=E_N$, a line bundle $\calL$ on
$E_{sN}$ with degree zero and non-periodic multidegree,  and a number $m\in \mathbb{N}$. By Lemmas
\ref{l:DirectInverse} and \ref{l:tensor}, $\cE$ is $H$-semistable
if and only if $\calL$ is $\pi_r^\ast H$-semistable. Using
Lemma \ref{l:rank1}, we see that the line bundle $\calL$ is
not stable because otherwise one would have
$\underline{d}(\calL)=(0,\dots,0)$ and this contradicts the
non-periodicity of $\underline{d}(\calL)$. Proceeding as above one sees that the graded object of $\calL\otimes\cF_m$ is $\oplus \cO_C(-1)^{\oplus m}$, where the sum runs over all the irreducible components $C$ of $E_{sN}$. Since $\pi_s$ is a finite morphism, one sees that $Gr(\cE)\simeq\oplus_{i=1}^N \cO_{C_i}(-1)^{\oplus r}$.
\end{proof}

This shows that in this case there exist indecomposable vector
bundles whose Jordan-H\"{o}lder factors are non-perfect.

\subsection{The biggest component of the moduli space}\label{ss:bigcomp}

In this subsection we describe completely the component
$\cM_X((\bar r,\dots, \bar r),0)\subset \cM_X(r,0)$ ($r=\bar rh$)
of semistable sheaves of multirank $(\bar r,\dots, \bar r)$ and
degree zero for the curve $E_N$ (cf.~Figure \ref{fig:en}) with
respect to an arbitrary polarization $H$.

For any smooth elliptic curve or a rational curve with one node or one cusp, it is well known
that the moduli space $\cM^{s}((1),0)$ is isomorphic to the curve. This is no longer true for
reducible fibers. In the particular case of $E_N$, the
moduli space $\cM((1,\dots,1),0)$ is isomorphic to a rational curve with a node $E_1$. This
was proved in \cite{Ca2} if $N=2$ and in \cite{LM05a} for any $N\geq 2$.

\begin{figure}[ht] 
 \centering
\includegraphics[width=1.5in]{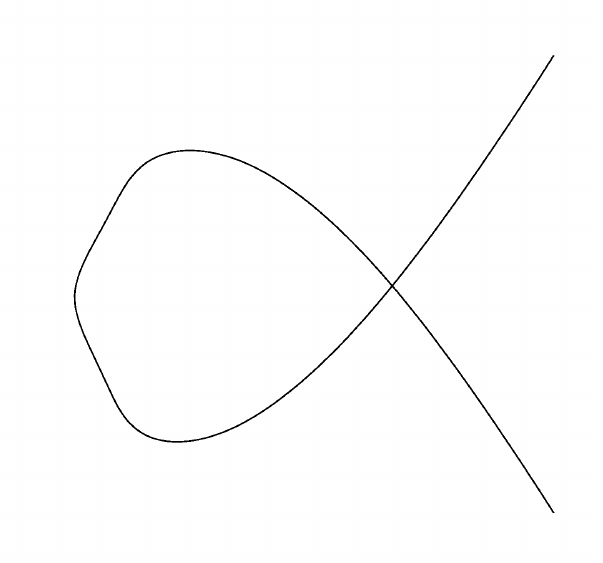} \caption{The rational curve with a node $E_1$} \label{fig:nodal}
\end{figure}

Let us describe the isomorphism $\cM((1,\dots,1),0)\iso E_1$,
following \cite[Proposition~3.2]{LM05a}, and the inverse isomorphism. Let
$C_1,\dots,C_N$ be the rational components of $X=E_N$ ordered
cyclically and let us denote by $\gamma\colon X \to E_1$ the
morphism which contracts $C_2,\dots,C_N$ and gives an isomorphism
$C_1-\{x_1,x_N\}\iso E_1- \{\bar z\}$, where $\{x_1,x_N\}$ are the
intersection points of $C_1$ with the other components and $\bar z$
is the singular point of $E_1$. Let us consider the sheaf on
$X\times X$
\begin{equation} \label{e:univ}
\cE=\cI_\Delta^\ast\otimes\pi_1^\ast\cO_X(-y_0)\,,
\end{equation}
where  $\cI_\Delta$ is the ideal sheaf of the diagonal immersion
$\delta$, $\pi_1$ is the canonical projection onto the first factor,
and $y_0$ is a fixed smooth point of $C_1$. For every point $y\in X$
the restriction $\cE_y$ of $\cE$ to $X\times\{y\}$ is a semistable
pure dimension one sheaf of multirank $(1,\dots,1)$ and degree $0$
\cite[Proposition~3.2]{LM05}. By \cite[Corollary~2.65]{BBH08}, the
Fourier-Mukai functor $\fmf{\cE}{X}{X}$ induces a classifying
morphism $\varphi\colon X\to \cM((1,\dots,1),0)$. Moreover,
for all points $y\in C_2\cup\dots\cup C_N$, the sheaves $\cE_y$ are
S-equivalent \cite[Proposition~3.2]{LM05a}, and then $\varphi$ factors
through a morphism
$$
\bar\varphi\colon E_1 \to \cM((1,\dots,1),0)\,,
$$
which one proves to be an isomorphism \cite{LM05a}. We shall give here a different proof (see Proposition \ref{p:isocase1}).

The smooth
points $y\in C_1$ are mapped to stable line bundles
$\mathfrak{m}_y^\ast \otimes \cO_X(-y_0)$ and all the remaining points
are mapped to the strictly semistable sheaves, all of them
S-equivalent to $\cO_{C_1}(-1)\oplus\dots\oplus \cO_{C_N}(-1)$.

Let us describe the action on $\dbc{X}$ of the quasi-inverse
$\fmf{\cE^\ast[1]}{X}{X}$ of $\fmf{\cE}{X}{X}$. Using the exact
sequence $0\to \cI_\Delta \to \cO_{X\times X}\to \delta_\ast\cO_X
\to 0$ and flat cohomology base-change, we get for every complex
$\cplx F$  in $\dbc{X}$ an exact triangle
$$
\fmf{\cE^\ast[1]}{X}{X}(\cplx F) \to p^\ast(\bR p_\ast(\cplx
F\otimes\cO_X(y_0))[1] \to \cplx F\otimes\cO_X(y_0)[1] \to
\fmf{\cE^\ast[1]}{X}{X}(\cplx F) [1]\,,
$$
where $p$ is the projection of $X$ onto one point. Applying this
formula to the Jordan-H\"older factors $\cO_{C_i}(-1)$ we have
\begin{align*}
\fmf{\cE^\ast[1]}{X}{X}(\cO_{C_1}(-1))& = \cI_{C_1}[1] \,,\quad \text{$\cI_{C_1}$ being the ideal sheaf of $C_1$ in $X$,}\\
\fmf{\cE^\ast[1]}{X}{X}(\cO_{C_j}(-1))& = \cO_{C_j}(-1)\,,\quad
\text{for $j>1$.}
\end{align*}
We now consider the integral functor $\dbc{X} \to \dbc{E_1}$
obtained as the composition of the  quasi-inverse
$\fmf{\cE^\ast[1]}{X}{X}$ of $\fmf{\cE}{X}{X}$ and the derived
push-forward $\bR\gamma_\ast\colon \dbc{X}\to \dbc{E_1}$. This is
the integral functor with kernel
\begin{equation}\label{e:kernel}
\cplx{K}=\bR (1\times \gamma)_\ast \cE^\ast [1]\,.
\end{equation}
\begin{lem} \label{l:inverse} If $\cF$ is a semistable sheaf on $X$ of  multirank $(1,\dots,1)$ and degree 0, then $\fmf{\cplx K}{X}{E_1}(\cF)=\cO_{z}$ for a uniquely determined point $z\in E_1$. Moreover $z$ is the singular point of $E_1$ if and only if $\cF$ is strictly semistable.
\end{lem}
\begin{proof}
If $\cF$ is stable, then $\cF=\fmf{\cE}{X}{X}(\cO_y)$ for a smooth
point $y\in C_1$, so that  $\fmf{\cplx
K}{X}{E_1}(\cF)=\cO_{\gamma(y)}$.  Let us now compute the image of
the Jordan-H\"older factors. If $j>1$, the restriction of $\gamma$
to $C_j$ factors through the singular point $\bar z$ of $E_1$, so
that
$$
\fmf{\cplx K}{X}{E_1}(\cO_{C_j}(-1))= \bR \gamma_\ast
(\cO_{C_j}(-1)) = 0\,.
$$
In the case $j=1$, the ideal  $\cI_{C_1}$ is supported on
$C_2\cup\dots \cup C_N$ and one has $H^0(X,\cI_{C_1})=0$, $\dim H^1(X,\cI_{C_1})=1$. As above, the
restriction of $\gamma$ to $C_2\cup\dots \cup C_N$ factors through
the singular point $\bar z$ of $E_1$ and one has
$$
\fmf{\cplx K}{X}{E_1}(\cO_{C_1}(-1))= \bR \gamma_\ast (\cI_{C_1})[1]
= \cO_{\bar z}\,.
$$
It follows that if $\cF$ is strictly semistable, by applying
$\fmf{\cplx K}{X}{E_1}$ to a Jordan-H\"older filtration, one obtains
$\fmf{\cplx K}{X}{E_1}(\cF)=\cO_{\bar z}$.
\end{proof}

\begin{prop} \label{p:isocase1} The integral functor $\fmf{\cplx K}{X}{E_1}$ induces a morphism
$$
\eta\colon \cM((1,\dots,1),0) \to E_1\,,
$$ which is the inverse of
$\bar\varphi$.
\end{prop}
\begin{proof} The morphism $\eta$  exists by  Lemma \ref{l:inverse} and \cite[Corollary~2.65]{BBH08}. One checks directly that $\eta\circ\bar\varphi$ coincides with the identity on the closed points. Since $E_1$ is separated, the subscheme $Z\hookrightarrow E_1$ of coincidences of $\eta\circ\bar\varphi$ and the identity is closed; moreover  the closed points are dense because $E_1$ is projective, and thus, $Z$ is topologically equal to $E_1$. Since $E_1$ is reduced, $Z$ is algebraically equal to $E_1$ as well, so that $\eta\circ\bar\varphi$ is equal to the identity. Taking into account that $\cM((1,\dots,1),0)$ is projective and reduced, the same argument proves that $\bar\varphi\circ \eta$ is the identity as well.
\end{proof}

Our next aim is to find the relationship between the moduli space $\cM((\bar r,\dots,\bar r),0)$ of semistable pure dimension one sheaves of
multirank $(\bar r,\dots,\bar r)$ and degree $0$ on a cycle $E_N$,  and the symmetric product $Sym^{\bar r} E_1$ of the rational curve with a node.


The following result is known but we could not find a suitable reference:

\begin{lem}\label{lem:reduced}
The moduli space $\cM((\bar r,\dots,\bar r),0)$ is reduced.
\end{lem}

A proof of Lemma \ref{lem:reduced} in the case $\bar r=2$ can be found in \cite[Huiti\^{e}me partie, Theor\`eme 18]{Ses82}. Seshadri shows also in \cite{Ses82} that the case $\bar r>2$ follows from  a property of certain determinantal varieties, that was not proved at the time; however the property was established later by Strickland \cite{Str82} and this completed the proof (cf.~also \cite{Se08}).

\begin{thm}\label{t:sym} Assume $\bar r>1$. There exists a scheme isomorphism
$$
\cM((\bar r,\dots,\bar r),0) \iso Sym^{\bar r} E_1\,.
$$
\end{thm}
\begin{proof}
Using a smooth point $y_0$ of $C_1$ as above, we
construct an isomorphism $\bar\varphi\colon E_1 \to
\cM((1,\dots,1),0)$. Since the direct sum of $\bar r$ semistable
sheaves of  multirank $(1,\dots,1)$ and degree 0 is semistable of
multirank $(\bar r,\dots,\bar r)$ and degree 0, we have a morphism
$$
E_1\times\overset{\overset{\bar r}{\smile}}{\cdots}\times E_1 \to
\cM((1,\dots,1),0)
\times\overset{\overset{\bar r}{\smile}}{\cdots}\times
\cM((1,\dots,1),0) \xrightarrow{\oplus}
\cM((\bar r,\dots,\bar r),0)\,.
$$
 This morphism factors through the $\bar r$-th symmetric product $Sym^{\bar r} E_1$ and then induces a morphism $\bar\varphi_{\bar r} \colon Sym^{\bar r} E_1\to \cM((\bar r,\dots,\bar r),0)$, which is one-to-one on closed points by Corollary \ref{c:JHfactors}.

If $\cF$ is a semistable sheaf of multirank $(\bar r, \dots, \bar r)$ and degree 0, by Corollary \ref{c:JHfactors}, its graded object with respect to a Jordan-H\"older filtration is
$Gr(\cF)=\calL_1^{\oplus a_1}\oplus\dots\oplus \calL_m^{\oplus a_m}\oplus (\oplus_{j=1}^{N}\mathcal{O}_{C_j}(-1))^{\oplus (\bar r-u)}$, where $\calL_i$ are stable line bundles of degree 0 and $u=a_1+\dots +a_m$. If $\cplx K$ is the kernel
defined by Equation \eqref{e:kernel}, using a Jordan-H\"older filtration of $\cF$ as in the proof of Lemma \ref{l:inverse}, one sees that
$$
\fmf{\cplx K}{X}{E_1}(\cF)=\cO_{Z_1}\oplus \dots \oplus \cO_{Z_m}\oplus \cO_{\bar Z}\,,
$$
where $Z_i$ is a zero dimensional closed subscheme of $E_1$ of length $a_i$ supported at a smooth point $z$ and $\bar Z$ is a zero dimensional closed subscheme  of length $\bar r-u$ supported at the singular point $\bar z$. Then $\fmf{\cplx K}{X}{E_1}(\cF)$ is a skyscraper sheaf of length $\bar r$ on $E_1$ and by \cite[Corollary 2.65]{BBH08} there exists a morphism
$$
\eta_r\colon \cM((\bar r,\dots,\bar r),0) \to Sym^{\bar r} E_1\,,
$$
which is the inverse of $\bar\varphi_{\bar r}$ on closed points. Since $Sym^{\bar r} E_1$ is projective and reduced, proceeding as in the proof of Proposition \ref{p:isocase1} we see that $\eta_r \circ \bar\varphi_{\bar r}\colon Sym^{\bar r} E_1 \to Sym^{\bar r} E_1$ is the identity. Since $ \cM((\bar r,\dots,\bar r),0)$ is reduced, a similar argument shows that $(\bar\varphi_{\bar r})\circ (\eta_r)\colon \cM((\bar r,\dots,\bar r),0) \to \cM((\bar r,\dots,\bar r),0)$ is the identity as well.
\end{proof}

\begin{rem} \label{r:othercases} Arguing as in Proposition \ref{prop:RR} and Corollary \ref{c:isom}, one gets that the equivalences $\Phi$ and $\Psi$ in Section 1 induce isomorphisms
\begin{align*}
\cM((\bar r,\dots,\bar r),d) & \simeq \cM((d-\bar r,\dots,d-\bar r),-d)\qquad \text{for $d > \bar r$} \\
\cM((\bar r,\dots,\bar r),d) & \simeq \cM((\bar r-d,\dots,\bar r-d),d)\qquad \text{for $d \leq \bar r $} \\
\cM((\bar r,\dots,\bar r),d)& \simeq \cM((\bar r,\dots,\bar r),\bar r h+d)\, .
\end{align*}
Thus, for any integers $\lambda,\mu\in \mathbb{Z}$, one gets that the moduli spaces $\cM((\bar r,\dots,\bar r),d)$ where $d=\lambda h \bar r$ and $\cM((r_0,\dots,r_0),d_0)$ where $r_0=\mu\lambda \bar r\pm \bar r$ and $d_0=\pm\lambda h \bar r$ are also isomorphic to the $\bar r$-th symmetric product $ Sym^{\bar r} E_1$ of the nodal curve.
\end{rem}
Using the results obtained so far, we give a complete description of
all moduli spaces $\cM_X(r,d)$ of semistable sheaves on the curve
$E_2$  (cf.~Figure \ref{fig:e2}) with respect to a polarization of
the minimum possible degree $h=2$.
Note that if $h=2$, the case  $2r_0/h\leq d_0< r_0$
in Corollary \ref{c:isom} is not possible. One then has:

\begin{cor}\label{c:casoE2} Let $X$ be a curve of type $E_2$ with a polarization $H$ of degree $h=2$, and let $(r,d)$ be a pair of integers with $r\geq 0$. The moduli space $\cM_X(r,d)$ of semistable sheaves with Hilbert polynomial $P(s)=rs+d$ on $X$ is isomorphic either to the $d_0$-th symmetric power $Sym^{d_0}(X)$ of the curve or to $\cM(r_0,0)$. Moreover, if $r_0$ is even, then the  biggest connected component of $\cM(r_0,0)$ is isomorphic to the  symmetric power $Sym^{r_0/2} E_1$ of the nodal curve $E_1$.

\end{cor}

\def\cprime{$'$}

\end{document}